\documentclass[preprint,11pt]{article}
\usepackage{amsfonts}
\usepackage{mathrsfs}
\usepackage{amssymb,amsfonts,amsmath}
\usepackage{graphics}
\usepackage{esint}
\usepackage{epstopdf}
\graphicspath{{./figures/}}
 \usepackage{float}
\usepackage{times}
\usepackage{graphicx,multicol,enumerate,subfigure}
\usepackage{color}
\usepackage{fullpage}
\usepackage{setspace}
\usepackage[english]{babel}
\usepackage[version=3]{mhchem}
\usepackage{amsthm}
\usepackage{verbatim, latexsym}
\usepackage{colortbl}
\usepackage{multirow}
\usepackage{marginnote}
\usepackage{comment}
\usepackage{amstext}
\usepackage{bm}
\usepackage{wrapfig}

\definecolor{black}{rgb}{0,0,0}

\definecolor{red}{rgb}{1,0,0}
\newcommand\red[1]{\textcolor{red}{#1}}
\definecolor{blue}{rgb}{0,0,1}
\newcommand\blue[1]{\textcolor{blue}{#1}}

\newtheorem{theorem}{Theorem}[section]
\newtheorem{lemma}[theorem]{Lemma}

\newtheorem{corollary}[theorem]{Corollary}
\newtheorem{remark}[theorem]{Remark}

\newcommand{\dy}{\mathrm{d}y}

\newcommand{\prnt}[1]{\left( #1 \right)}
\newcommand{\eps}{\varepsilon}

\newcommand{\norm}[1]{\left\|#1\right\|}

\newcommand{\normL}[2]{\norm{#1}_{L^2\prnt{#2}}}

\newcommand{\normI}[2]{\norm{#1}_{L^{\infty}\prnt{#2}}}

\newcommand{\ep}{\eps}

\def\Dbe{{\partial^\alpha_t}}

\newcommand{\normq}[2]{\norm{#1}_{L^{q}\prnt{#2}}}
\title{Homogenization of time-fractional diffusion equations with periodic coefficients}

\author{
Jiuhua Hu\thanks{Department of Mathematics, Texas A\&M University,
College Station, TX 77843-3368, USA. E-mail: jhu@math.tamu.edu} \, and Guanglian Li\thanks{Department of Mathematics, Imperial College London, Kensington, London SW7 2AZ, United Kingdom. E-mail: guanglian.li@imperial.ac.uk, lotusli0707@gmail.com.} 
}
\begin{document}
\maketitle
\begin{abstract}
We consider the initial boundary value problem for the time-fractional diffusion equation with a homogeneous Dirichlet boundary condition and an inhomogeneous initial data $a(x)\in L^{2}(D)$ in a bounded domain $D\subset \mathbb{R}^d$ with a sufficiently smooth boundary. We analyze the homogenized solution under the assumption that the diffusion coefficient $\kappa^{\epsilon}(x)$ is smooth and
periodic with the period $\epsilon>0$ being sufficiently small. We derive that its first order approximation has a convergence rate of $\mathcal{O}(\epsilon^{1/2})$ when the dimension $d\leq 2$ and $\mathcal{O}(\epsilon^{1/6})$ when $d=3$. Several numerical tests are presented to show the performance of the first order approximation.

\noindent{\bf Keywords:}
time-fractional diffusion; homogenization; 2-scale asymptotic expansion; first order approximation; error estimate
\end{abstract}

\section{Introduction}
Let $D$ be a bounded domain in $\mathbb R^d\,(d=1,2,3)$ with
a sufficiently smooth boundary $\partial D$. We consider a partial differential equation with the fractional derivative in time $t$, satisfying:
\begin{equation}
\label{eq:fine-model}
\left\{
\begin{aligned}
%\frac{\partial^{\alpha}
\Dbe u^\epsilon(x,t) & =\nabla\cdot \left(\kappa^{\eps}({x})\nabla u^{\epsilon}(x,t) \right) && \text{in } D,
\quad t\in \left( 0,T\right] \\
u^\epsilon& =0 && \text{on }\partial D, \;t\in \left( 0,T\right] \\
u^\epsilon(0) & =a(x)   && \text{in } D.
\end{aligned}
\right .
\end{equation}
Here, $0< \alpha <1$ is a given fixed parameter and the matrix $\kappa^{\eps}({x}):=\kappa(x/\epsilon)$ is periodic with period $\eps$. Let $(\kappa_{ij}(x))_{i,j=1}^{d}$ be symmetric with $\kappa_{ij} (x)\in C^{\infty}(D)$, $i,j=1, \ldots,d$. We assume for some constant $\mu>0$, there holds
\[
\kappa_{ij} (x)\xi_i\xi_j\geq \mu |\xi|^2 \text{ for all } \xi\in \mathbb R^d \text{ and } x\in D.
\]
The initial data $a(x)\in L^2(D)$ is a given macro-scale function and $T>0$ is a fixed value.
%Here, $h(x)$ and $g(x,t)$ are two macro-scale functions with proper regurality.

In the model problem \eqref{eq:fine-model}, $\Dbe w$ refers to the
left-sided Caputo fractional derivative of order  $\alpha$ %($0<\alpha<1$)
of the function $w(t)$, defined by (see, e.g. \cite[p. 91, (2.4.1)]{KilbasSrivastavaTrujillo:2006}
or \cite[p. 78]{Podlubnybook})
\begin{equation*}
  \Dbe w (t) = \frac{1}{\Gamma(1-\alpha)}\int_0^t \frac{1}{(t-s)^\alpha}w'(s)\, \mathrm{ds}.
\end{equation*}

Fractional diffusion equations were introduced in physics with the aim of describing diffusions in media with fractal geometry \cite{nigmatullin1986realization}. They have been applied to many fields, e.g., in engineering, physics, biology and finance. Their practical applications include electron transport in Xerox photocopier, visco-elastic materials, and protein transport in cell membranes \cite{scher1975anomalous,giona1992fractional,kou2008stochastic}. In this paper, we are concerned with the fractional diffusion problem \eqref{eq:fine-model} in a heterogeneous periodic media $\kappa^{\epsilon}(x)$, which is utilized in many important applications, for instance, porous media and composite material modeling. Most recently, periodic structures are utilized in metamaterials \cite{Vanel:2017:119/64002} to design novel materials. The main challenge in the classical numerical treatment of these applications is that it becomes prohibitively expensive and even intractable at the microscale as $\epsilon\to 0$. 

The goal of this paper is to construct an efficient numerical solver for \eqref{eq:fine-model} based on homogenization theory \cite{ref:Jikov.1991}. The main idea of homogenization is to obtain the effective or homogenized problem by solving $d$ cell problems, and then the corresponding homogenized solution $u_0$ serves as a good approximation to original unknown $u^{\epsilon}$ as the period $\epsilon\to 0$. To the best of our knowledge, there has been no such result for time-fractional diffusion problems so far. However, there are quite a few results on the parabolic equations, i.e., $\alpha=1$ in \eqref{eq:fine-model}, for the same setting, cf. \cite{pastukhova2010estimates,Zhikov2006}. Nevertheless, it is nontrivial to generalize the results for the parabolic equations to the time-fractional equations with the same technique, mainly due to the lack of the product rule for the fractional derivative. 

We prove in Theorem \ref{lem:R2} the error between the exact solution $u^{\epsilon}$ and its first order approximation $u^{\epsilon}_1$ is of $\mathcal{O}(\epsilon^{\min\{1/2,2/d-1/2\}})$ in the Bochner space $L^p([s,T]; H^{1}(D))$ for any $s\in (0,T]$ and $p\in (1,\frac{1}{\alpha})$ as $\epsilon\to 0$. Thus, the first order approximation $u^{\epsilon}_1$ achieves optimal convergence rate of $\mathcal{O}(\epsilon^{1/2})$ when $d\leq 2$ and $\mathcal{O}(\epsilon^{1/6})$ when $d=3$. For the latter case, it is unclear whether or not this rate is optimal. The proof of the result relies on the regularity of time-fractional diffusion problems \cite{Sakamoto-Yamamoto11} and the introduction of cut-off functions to handle boundary layers and initial data \cite{MR1420948,Zhikov2006}. When the initial data has a better regularity, e.g., $a(x)\in H^1_0(D)$, there is no need to introduce a proper cut-off function for $a(x)$. 

Furthermore, we present in Section \ref{sec:num} a number of numerical tests for $d=2$ to verify Theorem \ref{lem:R2}. Specifically, we test the cases when $\kappa(x)$ is smooth and discontinuous, and when $\kappa(x)$ admits large deviations. 
Our numerical tests demonstrate a higher convergence rate, namely, $\mathcal{O}(\epsilon)$, compared to the theoretical result. We also observe that a larger deviation or discontinuity in the coefficient $\kappa(x)$ results in a larger error.  

The remainder of this paper is organized as follows. In Section \ref{sec:2scale} we present the homogenized equation using two-scale asymptotic expansion for problem \eqref{eq:fine-model}, as well as regularity results on time-fractional diffusion problems. We then introduce auxiliary functions and define the first order approximation in Section \ref{first order approximation}. The main error estimate is derived therein. To verify our theoretical findings, we present in Section \ref{sec:num} extensive numerical tests with smooth (or nonsmooth) diffusion coefficient and with smooth (or nonsmooth) initial data $a(x)$. Finally, we conclude in Section \ref{sec:conclusion} with a summary and discussion of our results.

\section{Two-scale asymptotic expansion}\label{sec:2scale}
This section is concerned with the two-scale asymptotic expansion of  the solution $u^\epsilon$ to \eqref{eq:fine-model}. The approach is standard and can be found, e.g., in \cite{ref:Jikov.1991}. We recall the general procedure for the sake of completeness. 

First, analogous to the standard homogenization theory, we denote $y:=x/\epsilon$ as the fast variable, and $x$ is referred as the slow variable. Let $L^{2}_{\#}(Y):=\{ u\in L^{2}(Y):  u \text{ is $y$-periodic}\}$ with $Y$ being a unit cell in $\mathbb{R}^d$. Similarly, we can define $H^{1}_{\#}(Y)$. Denote $V_{\#}(Y):=\{v\in H^{1}_{\#}(Y):\langle v\rangle=0\}$. The notation $A\lesssim B$ denotes $A\leq C B$ for some constant $C$ independent of the microscale $\epsilon$. Throughout the paper, we follow the Einstein summation convention. 

Under the assumption that the fast variable $y$ and the slow variable $x$ are independent when $\epsilon\to 0$,
we seek for an asymptotic expansion of the solution $u^\epsilon (x,t)$ as follows:
\begin{equation}\label{eq:ansatz}
u^\epsilon(x,t) = u_0(x,y,t) + \epsilon u_1(x,y,t)
+\epsilon^2 u_2(x,y,t) + \cdots
\end{equation}
with the functions $u_j(x,y,t)\in H^{1}_{\#}(Y)$ being periodic in the fast variable $y$ with period 1. The leading order term $u_0(x,y,t)$ is referred as the homogenized solution and the following terms $\epsilon^{k} u_k(x,y,t)$ are the $k^{\rm th}$ order corrector for $u^\epsilon(x,t)$ for $k=1,2,\cdots$. 

Denote by $A^\epsilon$ the second order elliptic operator
\begin{equation*}
A^\epsilon = -
\frac{\partial}{\partial x_i} \left ( \kappa_{ij}\left (x/\epsilon
\right ) \frac {\partial}{\partial x_j} \right ).
\end{equation*}
%When differentiating a function $\phi(x,x/\epsilon )$ with respect
%to $x$, we have
%\[
%\frac{\partial }{\partial x_j} =
%\frac{\partial }{\partial x_j} + \frac{1}{\epsilon}
%\frac{\partial }{\partial y_j},
%\]
%where $y$ is evaluated at $y=x/\epsilon$.
With this notation, we can expand $A^\epsilon$ as follows
\begin{equation*}
A^\epsilon = \epsilon^{-2} A_1 + \epsilon^{-1} A_2
+\epsilon^{0} A_3 ,
\end{equation*}
where
\begin{eqnarray*}
A_1 &=& - \frac{\partial}{\partial y_i} \left ( \kappa_{ij}(y)
\frac {\partial}{\partial y_j} \right ), \\
A_2 &=& - \frac{\partial}{\partial y_i} \left ( \kappa_{ij}(y)
\frac {\partial}{\partial x_j} \right )
-\frac{\partial}{\partial x_i} \left ( \kappa_{ij}(y)
\frac {\partial}{\partial y_j} \right ), \\
A_3& =& - \frac{\partial}{\partial x_i} \left ( \kappa_{ij}(y)
\frac {\partial}{\partial x_j} \right ) \;.
\end{eqnarray*}
Substituting the expansions for $u^\epsilon$ and $A^\epsilon$
into the differential equation \eqref{eq:fine-model}, and equating the terms with the
same power of $\epsilon$, we get
\begin{eqnarray}
A_1 u_0 & =& 0, \label{eqn:1next}\\
A_1 u_1 &+ & A_2 u_0 = 0, \label{eqn:2next} \\
A_1 u_2 &+ & A_2 u_1 +A_3 u_0 = -  \Dbe u_{0}.
\label{eqn:3next}
\end{eqnarray}
Next we examine these equations one by one. Equation (\ref{eqn:1next}) is equivalent to seeking $u_0(x,y,t)\in H^{1}_{\#}(Y)$, satisfying
\begin{eqnarray*}
- \frac{\partial}{\partial y_i} \left ( \kappa_{ij}(y)
\frac {\partial}{\partial y_j} \right ) u_0(x,y,t)  = 0.
\end{eqnarray*}
The theory of second order elliptic PDEs  implies
that $u_0(x,y,t)$ is independent of $y$. Consequently, we obtain
\begin{align}\label{eq:u0Ny}
u_0(x,y,t) = u_0(x,t).
\end{align}
Since there is no micro-scale $\epsilon$ in either the boundary or the initial conditions, those conditions are imposed on the leading order term $u_{0}$ directly, i.e., 
\begin{equation}\label{u0:initial}
\begin{aligned}
u_{0}(0)&=a(x)&&\text{ in } D\\
u_{0}&=0&& \text{ on }\partial D.
\end{aligned}
\end{equation}

Thanks to \eqref{eq:u0Ny}, the second equation \eqref{eqn:2next} can be formulated as seeking for
$u_1\in H^{1}_{\#}(Y)$, satisfying
\[
- \frac{\partial}{\partial y_i} \left ( \kappa_{ij}(y)
\frac {\partial}{\partial y_j} \right )
u_1 = \left ( \frac{\partial}{\partial y_i} \kappa_{ij}(y)
\right ) \frac {\partial u_0}{\partial x_j}(x,t).
\]
Note that $u_1$ can be defined alternatively independent of the slow variable $x$. To this end,
let $\chi_j\in V_{\#}(Y)$ be the solution to the following {\it cell problem}:
\begin{equation}\label{eq:cellPrb}
 -\frac{\partial}{\partial y_i} \left ( \kappa_{ij}(y)
\frac {\partial}{\partial y_j} \right ) \chi_j =
\frac{\partial }{\partial y_i} \kappa_{ij}(y)  \; \text{ in } Y.
\end{equation}
The general solution of equation (\ref{eq:cellPrb}) for $u_1$
then admits the expression
\begin{equation}\label{eq:corrector}
u_1(x,y,t) = \chi_j(y) \frac{\partial u_0}{\partial x_j}(x,t) +
\tilde{u}_1(x,t) \; .
\end{equation}
For simplicity, we take $\tilde{u}_1(x,t)=0$.

Finally, we deal with the third term $u_2$. Because of \eqref{eqn:3next}, we can seek for $u_2\in H^{1}_{\#}(Y)$, such that,
\begin{equation}
\frac{\partial}{\partial y_i} \left ( \kappa_{ij}(y)
\frac {\partial}{\partial y_j} \right ) u_2 =
A_2 u_1 + A_3 u_0 +\Dbe u_{0}\;. \label{eqn:u2next}
\end{equation}
The solvability condition implies that the right hand side
of (\ref{eqn:u2next}) must have mean zero in $y$ over the unit cell $Y=[0,1]^d $, i.e. %$\times[0,1]$, i.e.
\[
\int_Y ( A_2 u_1 + A_3 u_0+\Dbe u_{0}) \,\dy = 0 .
\]
We note that
\[
\int_Y {\partial \over \partial y_i} F(x,y,t) \,\dy = 0
\]
for any $F(x,y,t)$ which is periodic with respect to $y$.
This can be easily verified using the Divergence Theorem. After integrating (\ref{eqn:u2next}) over $Y$, 
the average over the terms starting with ${\partial \over\partial y_i}$ disappears and we arrive at 
\[
\Dbe u_{0} - \frac{\partial}{\partial x_i} \left ( \langle \kappa_{ij}(y)\rangle
\frac {\partial}{\partial x_j} u_0 \right )  -
\frac{\partial}{\partial x_i} \left ( \langle \kappa_{ij}(y)
\frac {\partial}{\partial y_j} u_1\rangle \right ) =0.
\]
Substituting the expression for $u_1$ into this equation, we obtain the homogenized equation:
\begin{equation}\label{eqn:homnext}
 \Dbe u_{0}  -\frac{\partial}{\partial x_i} \left ( \kappa_{ij}^*
\frac {\partial}{\partial x_j} \right ) u_0
= 0 \;,
\end{equation}
where
\begin{equation}
\kappa_{ij}^* = \frac{1}{|Y|} \left ( \int_Y \kappa_{ij} +\kappa_{ik}
\frac{\partial \chi_j}{\partial y_k}  \,\dy \right ).
\label{eqn:coefnext}
\end{equation}
Last, we derive the {\em a priori} estimate for the homogenized solution $u_0$, which will be utilized below. We obtain from \eqref{u0:initial} and \eqref{eqn:homnext} that
\begin{equation}\label{eq:u0}
\left\{
\begin{aligned}
 \Dbe u_{0}  &=\frac{\partial}{\partial x_i} \left ( \kappa_{ij}^*
\frac {\partial}{\partial x_j} \right ) u_0&&\text{ in } D, \quad t\in \left( 0,T\right]\\
u_{0}&=0&& \text{ on }\partial D, \;\; t\in \left( 0,T\right]\\
u_{0}(0)&=a(x) &&\text{ in } D.
\end{aligned}
\right.
\end{equation}
Then by application of \cite[Theorem 2.1]{Sakamoto-Yamamoto11}, we obtain that $u_0\in C([0,T];L^2(D))\cap C((0,T];H^2(D)\cap H^1_0(D))$. Furthermore, the following estimate holds:
\begin{align}\label{eq:apriori1}
t^{\alpha}\norm{u_0(\cdot,t)}_{H^2(D)}+\norm{u_0(\cdot,t)}_{L^2(D)}+t^{\alpha}\normL{\Dbe u_0(\cdot,t)}{D}&\lesssim \norm{a}_{L^2(D)}, &&\text{ for all }t\in (0,T].%\\
%\normL{\nabla ^2 u_0(\cdot,t)}{D}&\lesssim t^{-\alpha}\normL{h}{D}&&\text{ for all }t\in (0,T].
\end{align}
Estimate \eqref{eq:apriori1} implies that the homogenized solution $u_0(x,t)$ is singular at $t=0$. The strength of this singularity depends on the regularity of the initial data $a(x)$. 
This singularity can disappear when the initial data $a(x)$ has certain regularity. Furthermore, the regularity of the solution $u_0$ from time-fractional diffusion problem has only limited regularity over the space domain $D$, c.f. \cite{jin2019numerical}. 
 
%Consequently, an application of the Interpolation Theory reveals
%\begin{align}\label{eq:apriori2}
%\normL{\nabla u_0(\cdot,t)}{D}&\lesssim t^{-\alpha/2}\normL{h}{D}&&\text{ for all }t\in (0,T].
%\end{align}
\begin{remark}
Note that \cite[Theorem 2.1]{Sakamoto-Yamamoto11} still holds when the permeability coefficient $\kappa^{\eps}$ admits high oscillation. Furthermore, the {\it a priori} estimate \eqref{eq:apriori1} is stable with respect to the parameter $\eps$. 
\end{remark}
\section{First order approximation estimate}\label{first order approximation}
We present in this section the first order approximation to $u^{\eps}$, and then derive its error estimate.

To this end, we will first provide the {\em a priori} estimate to the solutions of the cell problem  \eqref{eq:cellPrb}. This result can be found, e.g., in \cite[Section 8]{pastukhova2010estimates}. For the completeness, we also present the proof:
\begin{lemma}
Let $\chi_j$ be the solution to the cell problem  \eqref{eq:cellPrb} for all $j=1,\cdots,d$. Then there holds
\begin{align}
\norm{\nabla \chi_j}_{L^{2}(Y)}&\leq \mu^{-1} \norm{\kappa}_{L^{2}(Y)},&&\label{ineq:chi}\\
\|\chi_j\|_{H^{2}(Y)}+\|\chi_j\|_{L^{\infty}(Y)}&\lesssim 1.&&\label{eq:chi_Hm}
\end{align}
\end{lemma}
\begin{proof}
Let $e_j$ be the $j^{\rm th}$ canonical unit vector in $\mathbb{R}^d$. The weak formulation associated to Problem \eqref{eq:cellPrb} is to seek $\chi_j\in V_{\#}(Y)$ such that
\[
\int_Y \kappa(y) \nabla \chi_j\cdot \nabla v\,\dy=-\int_Y \kappa (y) e_j \cdot\nabla v\,\dy, \text{ for all } v\in V_{\#}(Y).
\]
Testing with $v:=\chi_j$, we obtain 
\[
\int_Y \kappa(y) \nabla \chi_j \cdot\nabla \chi_j\,\dy=-\int_Y \kappa (y) e_j  \cdot\nabla\chi_j\,\dy.
\]
Then the boundedness of $\kappa(y)$ and an application of the H\"{o}lder's inequality reveal the assertion \eqref{ineq:chi}.

Furthermore, by application of \cite{grisvard2011elliptic,MR0241822}, we obtain
\begin{align}\label{eq:chi_H2}
\|\chi_j\|_{H^2(Y)}\lesssim \|\nabla\cdot(\kappa\nabla\chi_j)\|_{L^2(Y)}+\|\Delta\chi_j\|_{L^2(Y)}\lesssim 1.
\end{align}
This, together with the weak maximum principle, yields the second assertion \eqref{eq:chi_Hm}.
%Since the coefficient $\kappa(y)\in C^{\infty}(Y)$, we can obtain
%\begin{align}\label{eq:chi_Hm}
%\|\chi_j\|_{H^m(Y)}\lesssim 1 \text{ for all }m\in \mathbb{N}_{+}.
%\end{align}
\end{proof}
Note that both the exact solution $u^{\epsilon}$ and the homogenized solution $u_0$ satisfy a homogeneous Dirichlet boundary condition.
However, the first order corrector $\epsilon \chi_j({x}/{\epsilon}) \frac{\partial u_0}{\partial x_j}(x,t)$ admits $\epsilon$-oscillation over the global boundary $\partial D$. To account for this fact, we must
adapt the corrector near the boundary. We employ an approach developed in \cite{MR1420948}.

To this end, we introduce the cut-off function $\zeta^{\ep}$ corresponding to $\partial \Omega$.
Here $\zeta^{\ep}=1$ on $\partial D$ and $\mbox{supp}(\zeta^{\ep})\subset \{x\in \bar{D}:\; \mbox{dist}(x,\partial D)\leq \ep\}$.
With regularity condition $\zeta^{\ep}\in C^{2}\bar{(D)}$, for $0\leq \ell \leq 2$, we have
\begin{align}
\label{eq:zetaestimate}
||D^{\ell}\zeta^{\epsilon}||_{L^{\infty}(D)}\lesssim \ep^{-\ell}.
\end{align}
In addition, using (\ref{eq:zetaestimate}) and the fact that $|\mbox{supp}(\zeta^{\eps})|= O(\eps)$, and
also by applying H\"older's inequality, we have the following estimate for derivatives of the cut-off function $\zeta^{\ep}$ in $L^{q}({D})$ given by
\begin{equation}
\label{eq:qzetaestimate}
 \normq{D^{\ell}\zeta^{\eps}}{D}\leq C|\mbox{supp}(\zeta^{\eps})|^{1/q}\normI{D^{\ell}\zeta^{\eps}}{D}\lesssim \eps^{1/q-\ell}.
\end{equation}
To avoid higher regularity condition on the initial data $a(x)$, we employ the trick in \cite{Zhikov2006} and introduce another cutoff function $\eta(t;\theta)$ in the time domain $[0,T]$ for any parameter $\theta\in (0,T]$. It is defined as follows:
\begin{equation*}
\eta(t;\theta)=
\left\{
\begin{aligned}
&0, &&\text{ for }t\leq \theta/2\\
&1, &&\text{ for }t\geq \theta\\
&\text{linear}&& \text{ for }t\in (\theta/2,\theta).
\end{aligned}
\right.
\end{equation*}
Let the first order approximation to $u^{\epsilon}(x,t)$, and its modification be defined by 
\begin{equation}\label{eq:firstapprox}
\left\{
\begin{aligned}
U_1^{\epsilon}(x,t)&:=u_0(x,t) +\epsilon\chi_j({x}/{\epsilon}) \frac{\partial u_0}{\partial x_j}\\
u_1^{\epsilon}(x,t;\theta)&:=u_0(x,t) +\epsilon\eta(t;\theta) (1-\zeta^{\ep})\chi_j({x}/{\epsilon}) \frac{\partial u_0}{\partial x_j}, \text{ for any }\theta\in (0,T].
\end{aligned}
\right.
\end{equation}
%Let 
%\begin{align*}
%w^{\epsilon}(x,s;t)&:=\eta(s;t)u^{\epsilon}(x,s;t)&&\\
%w_1^{\epsilon}(x,s;t)&:=\eta(s;t)u_1^{\epsilon}(x,s;t)&&
%\end{align*}
Then by definition, 
\begin{align*}
u_1^{\epsilon}(x,t;\theta)&=0&&\text{on }\partial D, t\in (0,T] \\
 u_1^{\epsilon}(x,0;\theta)& =a(x)&&\text{ in } D.
\end{align*}
Denote $u^\epsilon=u_1^{\epsilon}+R(x,t;\theta)$, where $R(x,t;\theta)$ is the residual between the exact solution $u^\epsilon(x,t)$ and its modified first order approximation. Then plugging this expression of $u^{\epsilon}$ into \eqref{eq:fine-model} yields
\begin{align*}
\Dbe \Big (u_1^{\epsilon}+R(x,t;\theta) \Big)
 =\nabla\cdot \Big(\kappa({x/\epsilon})\nabla (u_1^{\epsilon}+R(x,t;\theta)) \Big) \text{ in } D.
\end{align*}
Collecting the terms and using the homogenized equation \eqref{eqn:homnext},
the boundary and initial conditions for $u_{0}(x,t)$, we arrive at the following error equation
\begin{equation}\label{eq:residual}
\left\{
\begin{aligned}
\Dbe R &=\nabla\cdot(\kappa({x/\epsilon})\nabla R)+f_1&&\text{ in } D, \quad t\in \left( 0,T\right] \\
R&= 0 &&\text{ on }\partial D,\;t\in \left( 0,T\right]  \\
R(x,0;\theta)&=0  &&\text{ in } D.
\end{aligned}
\right.
\end{equation}
Here,
\[
f_1(x,t;\theta):=\nabla\cdot \big(\kappa({x/\epsilon})\nabla u_1^{\epsilon} \big)-\Dbe u_1^{\epsilon}.
\]
\begin{remark}[Regularity of the modified first order approximation \eqref{eq:firstapprox}]
Thanks to the condition $a\in  L^2(D)$, we can obtain $u_{1}^{\epsilon}(x,t;\theta)\in C((0,T]; H^1_0(D))$ for all $\theta\in (0,T]$. 
\end{remark}
A fundamental ingredient of  the homogenization theory  will be a proper {\em a priori} estimate for the residual $R(x,t;\theta)$.
%In order to produce such an estimate we first simplify the
%
To this end, we
first estimate the term $f_1$:
\begin{lemma}\label{lemma:meanValue}
For all $s\in (0,T]$, $f_1\in C([s,T];H^{-1}(D))$. Moreover, there holds 
\begin{equation*}%\label{termt1}
\begin{aligned}
\norm{f_1}_{H^{-1}(D)}\lesssim \epsilon^{\min\{1/2,2/d-1/2\}}t^{-\alpha}\|a\|_{L^2(D)} \text{ for all } t\in [s,T].
\end{aligned}
\end{equation*}
\end{lemma}
\begin{proof}
The following identity holds
\begin{equation*}
\begin{aligned}
f_1=\nabla\cdot(\kappa({x/\epsilon})\nabla u_1^{\epsilon})  - \Dbe u_1^{\epsilon}
 =  &\nabla\cdot(\kappa({x/\epsilon})\nabla u_1^{\epsilon})
   -\nabla\cdot(\kappa^{*}\nabla u_{0}) \\
&+\nabla\cdot(\kappa^{*}\nabla u_{0})-\Dbe u_1^{\epsilon}  \\
&:=T_{1}+T_{2}.
\end{aligned}
\end{equation*}
To estimate $\norm{f_1}_{H^{-1}(D)}$, we only need to estimate $\norm{T_1}_{H^{-1}(D)}$ and $\norm{T_2}_{H^{-1}(D)}$. To this end, we can further split the first term into the summation of the two terms
\[
T_1:=\nabla\cdot(\kappa({x/\epsilon})\nabla u_1^{\epsilon})-  \nabla\cdot(\kappa^{*}\nabla u_{0})
:=T_{1,1}+T_{1,2}
\]
with
\begin{align*}
T_{1,1}&:=\nabla\cdot\Bigg(\kappa(\frac{x}{\epsilon})\nabla \Big(u_0+\epsilon \eta(t;\theta)\chi_j(\frac{x}{\epsilon}) \frac{\partial u_0}{\partial x_j} \Big)\Bigg) -\nabla\cdot(\kappa^{*}\nabla u_{0})\\
T_{1,2}&:=-\epsilon\eta(t;\theta)\nabla\cdot\Bigg(\kappa(\frac{x}{\epsilon})\nabla \Big( \zeta^\epsilon\chi_j(\frac{x}{\epsilon}) \frac{\partial u_0}{\partial x_j} \Big) \Bigg).
\end{align*}
To estimate $T_{1,1}$, we apply the argument in \cite[Section 1.4]{ref:Jikov.1991}, together with the definition of the cutoff function $\eta(t;\theta)$, and obtain 
\begin{equation}\label{termt11}
\begin{aligned}
\norm{T_{1,1}}_{H^{-1}(D)}\lesssim {\epsilon}t^{-\alpha/2}\|a\|_{L^2(D)} \text{ for any } \theta\in (0,T] \text{ and } t\geq\theta.
\end{aligned}
\end{equation}
Then we estimate $T_{1,2}$. A direct calculation results in
\begin{equation*}%\label{termt1_2}
\begin{aligned}
\norm{T_{1,2}}_{H^{-1}(D)}:=&\norm{-\epsilon\eta(t;\theta)\nabla\cdot\Bigg(\kappa(\frac{x}{\epsilon})\nabla \Big( \zeta^\epsilon\chi_j(\frac{x}{\epsilon}) \frac{\partial u_0}{\partial x_j} \Big) \Bigg)}_{H^{-1}(D)}\\
\leq & \epsilon\norm{\kappa(\frac{x}{\epsilon})\nabla \Big( \zeta^\epsilon\chi_j(\frac{x}{\epsilon}) \frac{\partial u_0}{\partial x_j} \Big) }_{L^{2}(D)}.
%\leq & \epsilon\norm{\kappa({x/\epsilon})(\nabla  \zeta^\epsilon)\chi_j({x/\epsilon}) \frac{\partial u_0}{\partial x_j}  }_{L^{2}(D)}
%+ \epsilon\norm{\kappa({x/\epsilon})\zeta^\epsilon\nabla (\chi_j({x/\epsilon}) \frac{\partial u_0}{\partial x_j} ) }_{L^{2}(D)} \\
\end{aligned}
\end{equation*}
Then by the triangle inequality and the chain rule, we deduce
\begin{align*}
\norm{T_{1,2}}_{H^{-1}(D)}
\lesssim \epsilon \|\nabla  \zeta^\epsilon  \|_{L^{2}(D)}\Big\|\chi_j(\frac{x}{\epsilon}) \frac{\partial u_0}{\partial x_j}\Big\|_{L^{2}(D)}
&+ \epsilon\Big\| \zeta^\epsilon\chi_j(\frac{x}{\epsilon}) \nabla\frac{\partial u_0}{\partial x_j}\Big\|_{L^{2}(D)}\\
&+ \Big\|\zeta^\epsilon\nabla \chi_j({\frac{x}{\epsilon}}) \frac{\partial u_0}{\partial x_j}\Big\|_{L^{2}(D)}.
\end{align*}
This, together with the generalized H\"{o}lder's inequality and Sobolev embedding, leads to
\begin{align*}
\norm{T_{1,2}}_{H^{-1}(D)}
\lesssim & \epsilon^{1/2}\|\chi_j\|_{L^{\infty}(Y)}\norm{ \frac{\partial u_0}{\partial x_j}  }_{L^{2}(D)}
+ \epsilon\norm{ \zeta^\epsilon}_{L^{2}(D)} \norm{\chi_j}_{L^{2}(Y)} \norm{ u_0}_{H^{2}(D)}\\
&\hspace{4,5cm}+\norm{ \zeta^\epsilon}_{L^{q}(D)} \norm{\nabla \chi_j}_{L^{r}(Y)} \norm{\frac{\partial u_0}{\partial x_j}  }_{L^r(D)}
\end{align*}
where
\begin{equation*}
\frac{1}{r}=
\left\{
\begin{aligned}
&0, &&\text{ when } d=1\\
&\delta, &&\text{ when } d=2\\
&\frac{1}{2}-\frac{1}{d}, &&\text{ when } d=3
\end{aligned}
\right.
\qquad\text{and} \qquad
\frac{1}{q}=
\left\{
\begin{aligned}
&\frac{1}{2}, &&\text{ when } d=1\\
&\frac{1}{2}-2\delta, &&\text{ when } d=2\\
&\frac{2}{d}-\frac{1}{2}, &&\text{ when } d=3.
\end{aligned}
\right.
\end{equation*}
Here, the parameter $\delta\in (0,1/2)$ is an arbitrary constant, and $d$ is the dimension of the domain $D$.

Together with the estimates \eqref{ineq:chi}, \eqref{eq:qzetaestimate} and \eqref{eq:apriori1}, by letting $\delta\to 0$, we obtain
\begin{align}\label{termt12}
\norm{T_{1,2}}_{H^{-1}(D)}
\lesssim \eps^{\min\{1/2,2/d-1/2\}}t^{-\alpha}\norm{a}_{L^{2}(D)}.
\end{align}
Next we estimate the second term $T_2$. The estimate \eqref{eq:u0} implies
\begin{align*}
T_{2}&:=\nabla\cdot \big(\kappa^{*}\nabla u_{0} \big)-  \Dbe u_1^{\epsilon}
= \eta(s;t)\Dbe (u_{0}-u_1^{\epsilon})\\
&=\epsilon (\zeta ^{\epsilon}-1) \chi_j(\frac{x}{\epsilon})\Dbe  \frac{\partial u_0}{\partial x_j}.
\end{align*}
 By exchanging the fractional derivative with respect to time $t$ and the derivative with respect to the space variable $x$,
we obtain
\begin{align*}
T_{2}=\epsilon (\zeta ^{\epsilon}-1)  \chi_j(\frac{x}{\epsilon})\frac{\partial }{\partial x_j}\Dbe u_0.
\end{align*}
Then by application of \eqref{eq:apriori1}, we arrive at
\begin{align*}
\|T_2\|_{H^{-1}(D)}\lesssim \epsilon t^{-\alpha}\|a\|_{L^2(D)}.
\end{align*}
This, together with \eqref{termt11}, \eqref{termt12} and an application of the triangle inequality,  proves the desired assertion after letting $\theta\to s$.
\end{proof}

Finally, we are ready to present the error estimate for the residual $R$ defined in the error equation \eqref{eq:residual}:

\begin{theorem}[Modified first order approximation]\label{lem:R2}
For any $p\in [1,\frac{1}{\alpha})$ and $s\in (0,T]$, there exists a unique weak solution
$R\in {L}^{p}([s,\,T]; {H}^{1}_{0}(D))$ to Problem \eqref{eq:residual} such that $\Dbe R\in {L}^{p}([s,\,T]; {H}^{-1}(D))$, satisfying
\begin{align}\label{estimate:R_2}
\norm{R}_{L^{p}([s,\,T]; H^{1}(D))}+\norm{\Dbe R}_{{L}^{p}([s,\;T]; {H}^{-1}(D))}
\lesssim \epsilon^{\min\{1/2,2/d-1/2\}}\|a\|_{L^2(D)}.
\end{align}
%In particular, for any $\gamma\geq {d\over 4}-{1\over 2}$, we have $R\in \mathcal{C}([0,\;T], \mathcal{D}((-L)^{-\gamma}))$.
\end{theorem}
\begin{proof}
First, Lemma \ref{lemma:meanValue} gives
\[
\norm{f_1}_{{H}^{-1}(D)}\lesssim \epsilon^{\min\{1/2,2/d-1/2\}}t^{-\alpha}\|a\|_{L^2(D)}\quad\text{ for all }t\in \left[s,T\right].
\]
Therefore, $f_1\in L^{p}([s,T]; H^{-1}(D))$ for any $p\in [1,\frac{1}{\alpha})$, and there holds
\[
\norm{f_1}_{L^p([s,T];{H}^{-1}(D))}\lesssim \epsilon^{\min\{1/2,2/d-1/2\}}\|a\|_{L^2(D)}.
\]
Meanwhile, following the proof of \cite[Theorems 2.1 and 2.2]{Sakamoto-Yamamoto11} and \cite[Theorem 2.3]{JLPZ14}, we can arrive at
\begin{align*}
\norm{R}_{L^{p}([s,\;T]; H^{1}(D))}+\norm{\Dbe R}_{{L}^{p}([s,\;T]; {H}^{-1}(D))}
\lesssim \norm{f_1}_{{L}^{p}([s,\;T]; {H}^{-1}(D))}.
\end{align*}
Consequently, a combination of the previous two estimates proves the desired assertion.
\end{proof}
We present in the next result the error estimate with the boundary layer effect:
%===============================================================
\begin{corollary}[First order approximation]\label{coro:1order}
Let $a(x)\in L^2(D)$, $p\in [1,\frac{1}{\alpha})$, $s\in (0,T]$ and $\epsilon$ sufficiently small. Let $u^{\epsilon}$ be the solution to Problem \eqref{eq:fine-model}, then the following estimate holds 
\begin{align*}%\label{estimate:R_2}
\norm{u^{\epsilon}-U_1^{\epsilon}(x,t)}_{L^{p}([s,\,T]; H^{1}(D))}+\norm{\Dbe \Big( u^{\epsilon}-U_1^{\epsilon}(x,t)\Big)}_{{L}^{p}([s,\;T]; {H}^{-1}(D))}\lesssim \epsilon^{\min\{1/2,2/d-1/2\}}\|a\|_{L^2(D)}.
\end{align*}
\end{corollary}
\begin{proof}
The proof follows from the argument in Theorem \ref{lem:R2} and the fact that the estimate \eqref{termt12} dominates the convergence rate. 
\end{proof}
%=================================================================
\begin{comment}
Finally, we establish the {\em a priori} estimate to \eqref{eq:residual} with the help of the two results above:
\begin{theorem}\label{thm:estimate:R}
Let $R$ be the solution to Problem \eqref{eq:residual}. Then it holds $R\in {L}^{p}([0,\,T]; {H}^{1}_{0}(D))$ and $\Dbe R\in {L}^{p}([0,\,T]; {H}^{-1}(D))$ for any $p\in [1,\frac{1}{\alpha})$. Furthermore, the following estimate holds
\begin{align}\label{estimate:R}
\norm{R}_{L^{p}([0,\;T]; H^{1}(D))}+\norm{\Dbe R}_{{L}^{p}([0,\;T]; {H}^{-1}(D))}
\lesssim \epsilon^{\min\{1/2,2/d-1/2\}}\|a\|_{L^2(D)}.
\end{align}
\end{theorem}
\begin{proof}
The combination of Lemma \ref{lem:R1} and Interpolation theory implies 
\begin{align*}
 \norm{R_1(t)}_{ H^{1}(D)}
&\lesssim\eps t^{-\frac{\alpha}{2}}\norm{a}_{{H}^{1}(D)} \quad\text{ for all }t\in \left( 0,T\right]. 
\end{align*}
Then taking $L^p([0,T])$ norm over both sides yields 
\begin{align*}
\norm{R_1}_{ L^{p}([0,T];H^{1}(D))}
&\lesssim\eps \norm{a}_{{H}^{1}(D)}.
\end{align*}
Then an application of the triangle inequality and Lemma \ref{lem:R2} proves the desired assertion.
\end{proof}
\end{comment}
\begin{remark}[Comparision with homogenization results to parabolic equations with periodic coefficients]
Due to the boundary layer effect in the bounded domain, we can obtain from \cite{pastukhova2010estimates,Zhikov2006} that the optimal convergence rate is $\mathcal{O}(\epsilon^{1/2})$. Corollary \ref{coro:1order} shows that the convergence is $\mathcal{O}(\epsilon^{\min\{1/2,2/d-1/2\}})$. One main restriction to apply similar technique employed in \cite{Zhikov2006} to our current problem of a time-fractional diffusion problem \eqref{eq:fine-model} is the fact that there is no product rule for fractional derivative, but there is for the first derivate. For the same reason, one can not derive the pointwise estimate over the time domain for time-fractional diffusion problems. 
\end{remark}
%\red{Is our result optimal? Can we derive uniform estimate on time?}
%\section{Second order approximation estimate}
\section{Numerical experiments}\label{sec:num}
\begin{comment}
The numerical experiments contains the following results
\begin{enumerate}
\item We use $L^2$ norm and $H^1$ norm for the error between $u^\epsilon$ and $u^\epsilon_1$
\item The error for the smooth permeability is smaller than that for the nonsmooth permeability;
\item The error grows as the contrast of the smooth permeability grows. We run the code for contrast=$\frac{11}{9}$ and $19$;
\item The error grows as the contrast of the nonsmooth permeability grows. We run the code for contrast=$1.1$ and $2$;
\item The error decreases as we decrease $\epsilon$, we run the code for $\epsilon=\frac{1}{2}, \frac{1}{8},\frac{1}{20}$ under the smooth permeability;
\item When varying $\alpha$ between 0 and 1, there is no significant change in the error. Here we run the code for smooth permeability
\end{enumerate}
\end{comment}

In this section, we conduct a series of numerical experiments to demonstrate the performance of the first order corrector introduced in Section \ref{first order approximation}. Furthermore, we will validate the convergence result presented in Corollary \ref{coro:1order} corresponding to different permeability fields and fractional power $\alpha$. 

Consider the time-fractional diffusion equation (\ref{eq:fine-model}) in the unit square $D=[0,1]^2$ with total time $T=1$ and $\alpha:=0.9$. We will use scalar coefficient $\kappa^\epsilon(x_1,x_2)$ in the following numerical tests. The smooth initial data tested in Sections \ref{subsec:smooth} and \ref{subsec:nonsmooth} is 
\[
a(x_1,x_2):=x_1(1-x_1)x_2(1-x_2). 
\]
We refer to Figure \ref{fig:initial_data} for an illustration. 
\begin{figure}[H]
  \centering
  \includegraphics[trim={2cm 0.5cm 2cm 0.5cm},clip,width=0.45 \textwidth]{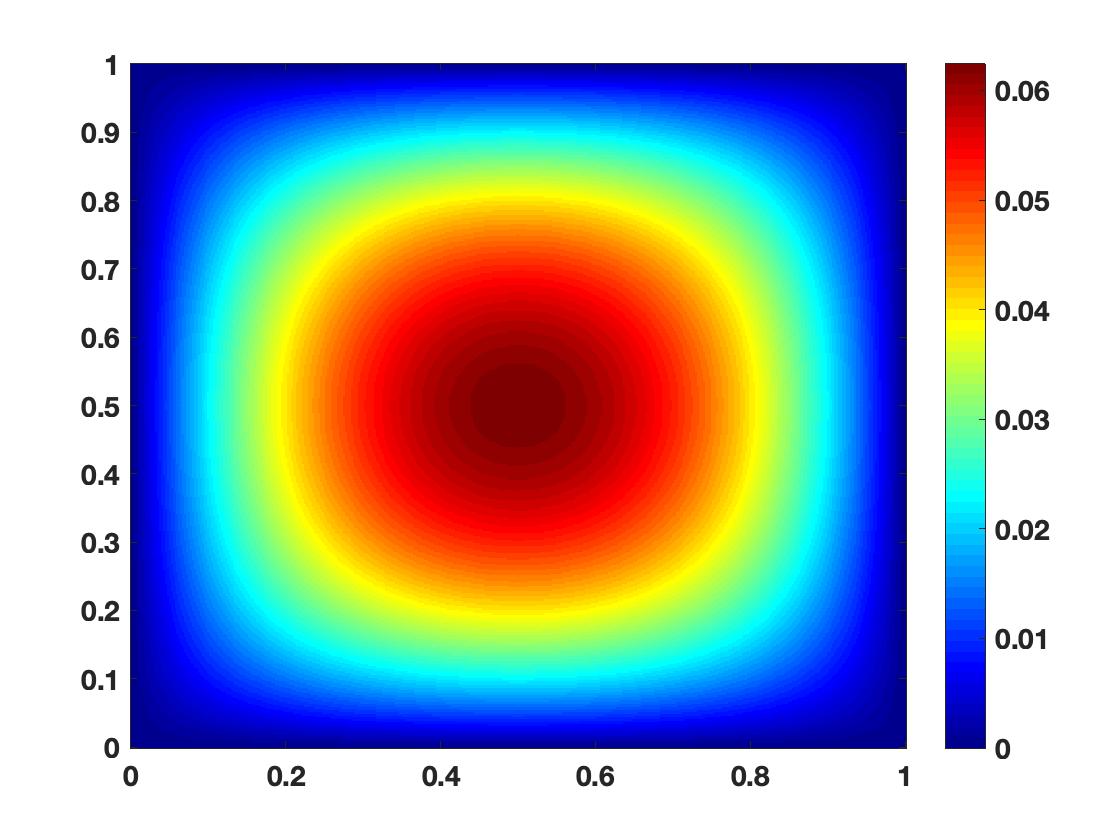}
  \caption{initial data: $a(x_1,x_2)$}
  \label{fig:initial_data}
\end{figure}

In Sections \ref{subsec:smooth} and \ref{subsec:nonsmooth}, we are concerned with the convergence rate of the first order approximation $U_1^\epsilon$ for diffusion coefficients $\kappa^\epsilon(x_1,x_2)$ of different regularities. To this end, we test two kinds of permeability fields, namely, the smooth and nonsmooth permeability fields in these two sections, respectively. Since Corollary \ref{coro:1order} is also valid for rough initial data $a(x)\in L^2(D)$, we present the convergence history of the first order approximation $U_1^\epsilon$ with a rough initial data in Section \ref{subsec:initial_L2}. 
\subsection{Numerical tests with smooth permeability fields}\label{subsec:smooth}
To define the smooth permeability field, we take
\begin{align}\label{eq:smoothPerm}
\kappa(y_1,y_2):=10+\sin\Big(2\pi\{ y_1\}\{ y_2\}\big (1-\{ y_1\}\big)\big(1-\{ y_2\}\big)\Big)
\end{align}
as the periodic smooth function defined over the unit square $Y$. 
Here, $\{\cdot\}$ means taking the fractional part. 
\begin{figure}[H]
  \centering
  \includegraphics[trim={2cm 0.5cm 2cm 0.5cm},clip,width=0.45 \textwidth]{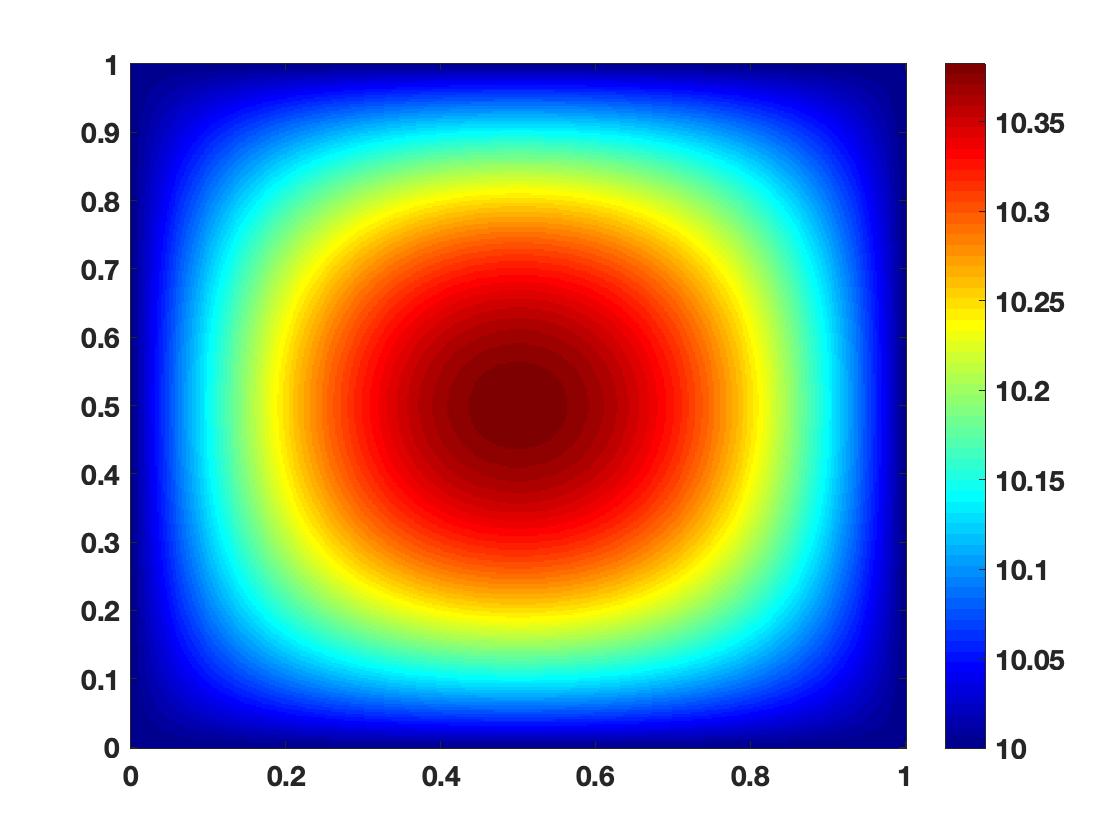}
  \includegraphics[trim={2cm 0.5cm 2cm 0.5cm},clip,width=0.45 \textwidth]{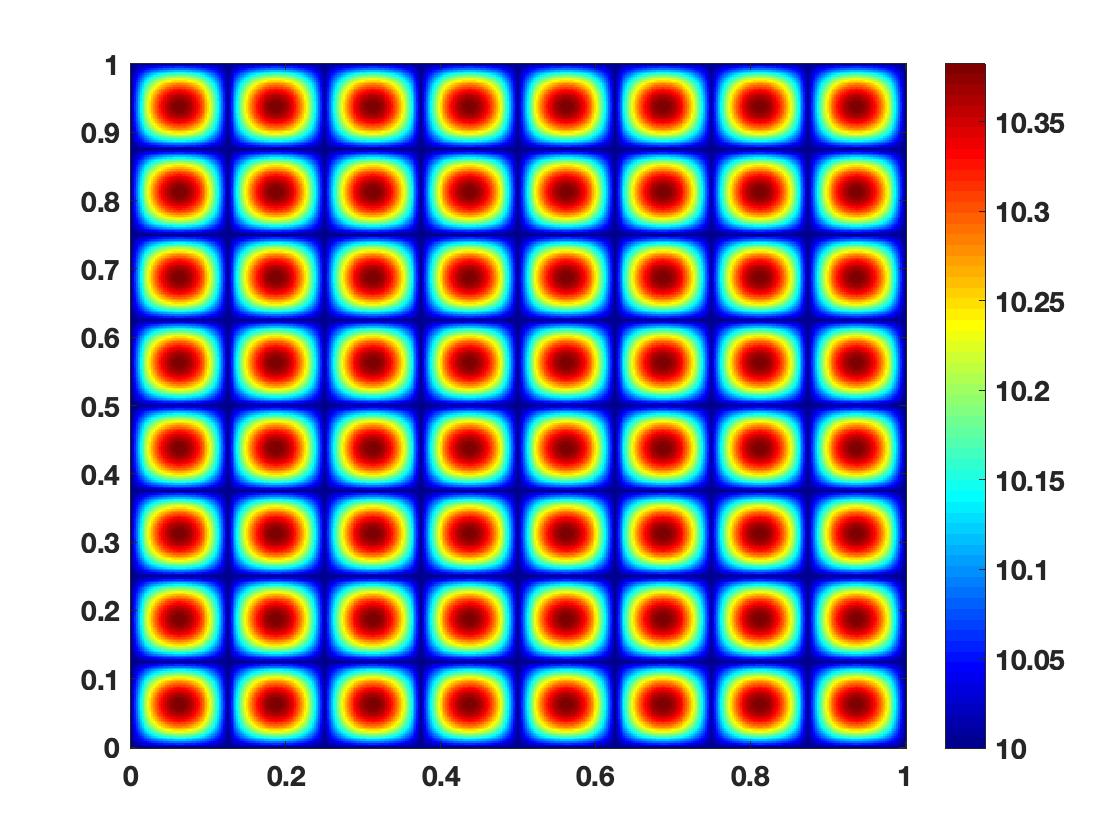}
 \caption{A smooth periodic permeability field with $\epsilon=\frac{1}{8}$ in a cell and over the domain $D$: $\kappa(y_1,y_2)$ and $\kappa^\epsilon(x_1,x_2)$.}
  \label{fig:SmoothPermeability}
\end{figure}

Recall that $\kappa^\epsilon(x_1,x_2)$ is a periodic function with period $\epsilon$. Its one cell after stretching over one unit cell is $\kappa(y_1,y_2)$, i.e., $\kappa^\epsilon(x_1,x_2):=\kappa(x_1/\epsilon,x_2/\epsilon)$. See Figure \ref{fig:SmoothPermeability} for an illustration.
Note that the contrast for this permeability is $\frac{11}{9}$.
The main aim of this section is to investigate the convergence rate of the first order corrector $U_1^\epsilon$. We will present the absolute and relative errors between $U_1^\epsilon$ and $u^\epsilon$ in $L^2$-norm and $H^1$-norm.

Let $\mathcal{T}_h$ be a decomposition of the domain $D$ into non-overlapping shape-regular rectangular elements with maximal mesh size $h:=2^{-9}$. Let $V_h$ be the conforming piecewise affine finite element space associated with the partition $\mathcal{T}_h$:
\[
V_h:=\{v\in C^{0}(D): v|_{T}\in \mathcal{Q}_{1}(T) \text{ for all } T\in \mathcal{T}_h\}\cap H^1(D),
\]
where $\mathcal{Q}_1(T)$ denotes the space of affine polynomials on each element $T\in \mathcal{T}_h$.
Let $V_h^{0}:=V_h\cap H^1_{0}(D)$.
%We will apply the conforming Galerkin approximation to solve for $p_h(\cdot,t)\in V_h$ with $p_h(\cdot,t)=g_D$ on $\Gamma_D$, s.t.,
%\begin{align}\label{eqn:weakform_h}
%\int_{\Omega}\kappa \nabla p_h(\cdot,t)\cdot \nabla v_h=0 \quad \text{ for all } v_h\in V_h^{0}.
%\end{align}
%We will apply the Euler implicit scheme to solve for the tensor $\kappa$ and the scalar $\rho$. Firstly,
We discretize the time interval $[0,1]$ with a time step $\Delta t:=1/100$. Let $t_n=(n-1)\Delta t$ for $n=1,2,\cdots,N$ with $N:=\Delta t^{-1}+1$.

%The time interval $[0,1]$ is divided into equal subintervals with equal time step size $\Delta t:=1/100$ and the corresponding time steps are $0=t_0<t_1<\cdots<t_{100}=1$, here $t_k=k\Delta t$, for $k=0,1,\cdots,100$.

We adopt one popular scheme \cite{chuanjuxu2007} to discretize the time variable $t$ in (\ref{eq:fine-model}) and (\ref{eq:u0}), and apply the conforming Galerkin method to approximate the exact solution $u^{\epsilon}$. We denote its approximation at $t=t_{k+1}$ by $u^{\epsilon,k+1}_h$ for $k=0,1,\cdots, N-1$. 

To this end, we seek for $u^{\epsilon,k+1}_h\in V_h^{0}$ for $k=0,1,\cdots, N-1$, satisfying
\begin{equation*}
\begin{aligned}
\forall \, v_h\in V_h^0:\int_{D}u^{\epsilon,k+1}_h v_h \,\mathrm{dx} &+\Gamma(2-\alpha)\Delta t^{\alpha}\int_{D}\kappa(x/\epsilon)\nabla u^{\epsilon,k+1}_h \cdot\nabla v_h \,\mathrm{dx}\\
&=\sum_{j=0} ^{k-1} (b_j-b_{j+1}) \int_{D}u^{\epsilon,k-j}_h v_h \,\mathrm{dx}
+b_k \int_{D}a(x) v_h \,\mathrm{dx} .
\end{aligned}
\end{equation*}

\begin{comment}
\begin{equation*}
\label{eq:apprximate_fine-model}
\left\{
\begin{aligned}
%\frac{\partial^{\alpha}
u^{\epsilon,k+1}_h -\alpha\nabla\cdot \left(\kappa(x/\epsilon)\nabla u^{\epsilon,k+1}_h \right)\\&=
\sum_{j=0} ^{k-1} (b_j-b_{j+1}) u^{k-j}_h+b_k u^0_h && \text{in } D 
u^{\epsilon,k+1}_h& =0 && \text{on }\partial D.
\end{aligned}
\right .
\end{equation*}
\end{comment}
Here, the parameters $b_j$ are given by 
\[
b_j:=(j+1)^{1-\alpha}-j^{1-\alpha}\text{ for all } j=0,1,\cdots,N.
\]
%We use the continous piecewise bilinear Lagrange Finite Element Method to approximate $u^{\epsilon,k+1}(x)$.

The numerical solutions $u^{\epsilon,k}_h$ for $k=11,51 \text{ and } 101$ are depicted in Figure \ref{fig:u_eps}.
\begin{figure}[H]
  \centering
  \includegraphics[trim={2cm 0.5cm 2cm 0.5cm},clip,width=0.32 \textwidth]{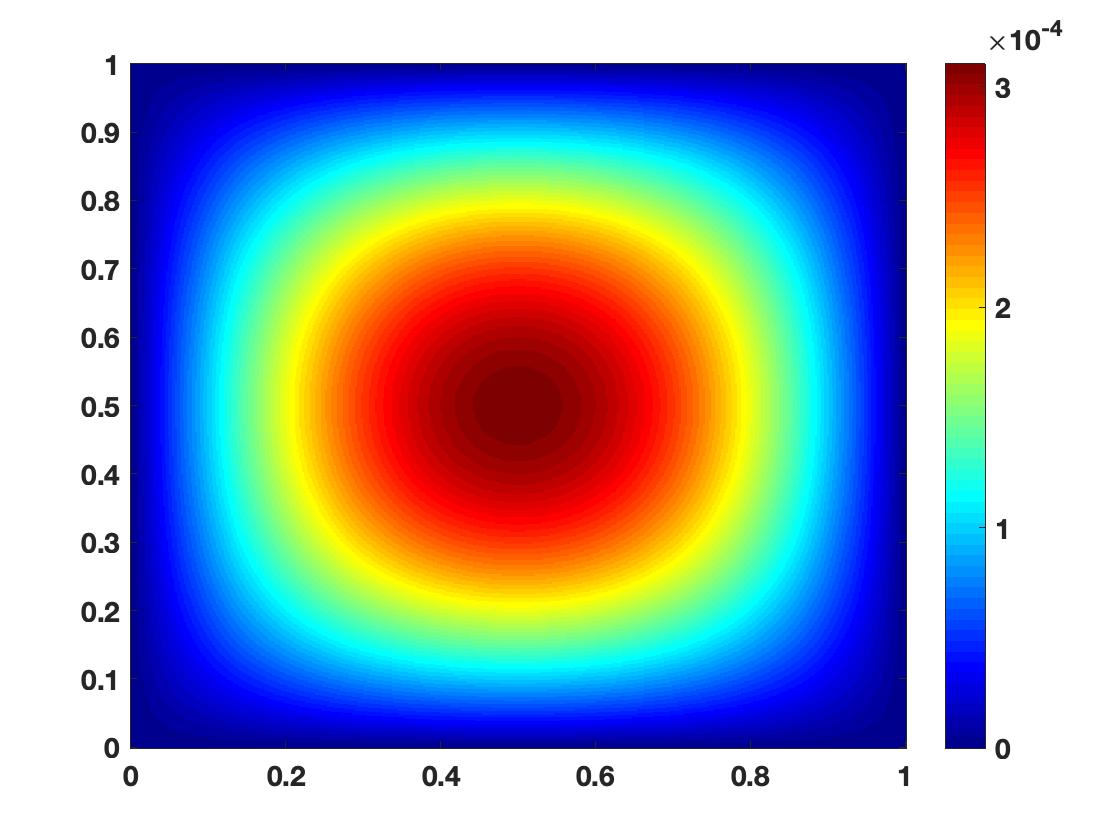}
  \includegraphics[trim={2cm 0.5cm 2cm 0.5cm},clip,width=0.32 \textwidth]{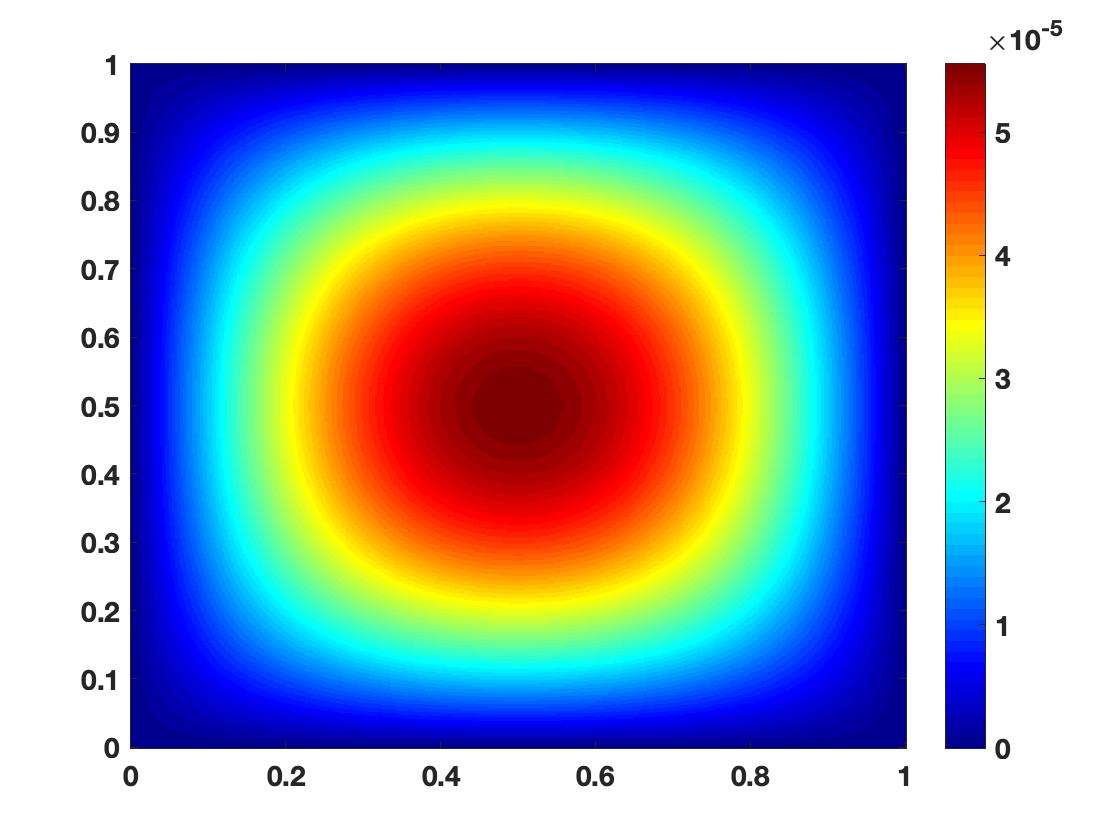}
  \includegraphics[trim={2cm 0.5cm 2cm 0.5cm},clip,width=0.32 \textwidth]{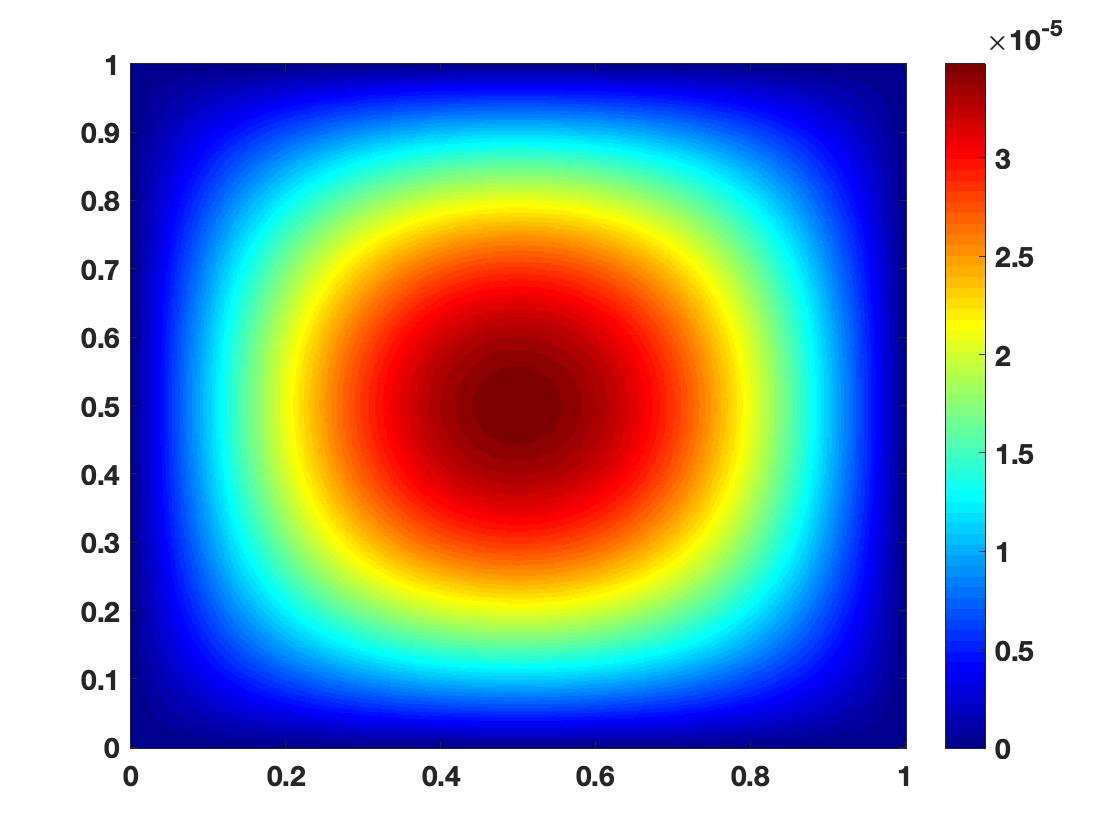}
  \caption{The fine scale solution $u^{\epsilon,k}_h$ for $k=11,51 \text{ and } 101$ to Problem \eqref{eq:fine-model} with $\kappa^{\eps}$ in \eqref{eq:smoothPerm} and $\epsilon:=\frac{1}{8}$. }
  \label{fig:u_eps}
\end{figure}

We denote $U^{\epsilon,k}_{1,h}$ (or $u_{0,h}^{k}$) for $k=1,2\cdots,N$ as the numerical approximation to $U^{\epsilon}_{1}(t)$ (or $u_{0}(t)$) for $t=0,\Delta t, \cdots,1$. In order to obtain the first order approximation $U^{\epsilon,k}_{1,h}$, we first solve the cell problem (\ref{eq:cellPrb}). To this end, we first divide the computational domain $Y$ into non-overlapping shape-regular rectangular elements with a maximal mesh size $h:=2^{-6}$. Then we solve the cell problem \eqref{eq:cellPrb} with continuous piece-wise bilinear Lagrange Finite Element Method using the conforming Galerkin formulation. We plot the two cell solutions $\chi_1$ and $\chi_2$ in Figure \ref{fig:chi_smooth} with the smooth permeability field in Figure \ref{fig:SmoothPermeability}.
\begin{figure}[htb!]
  \centering
  \includegraphics[width=0.45 \textwidth]{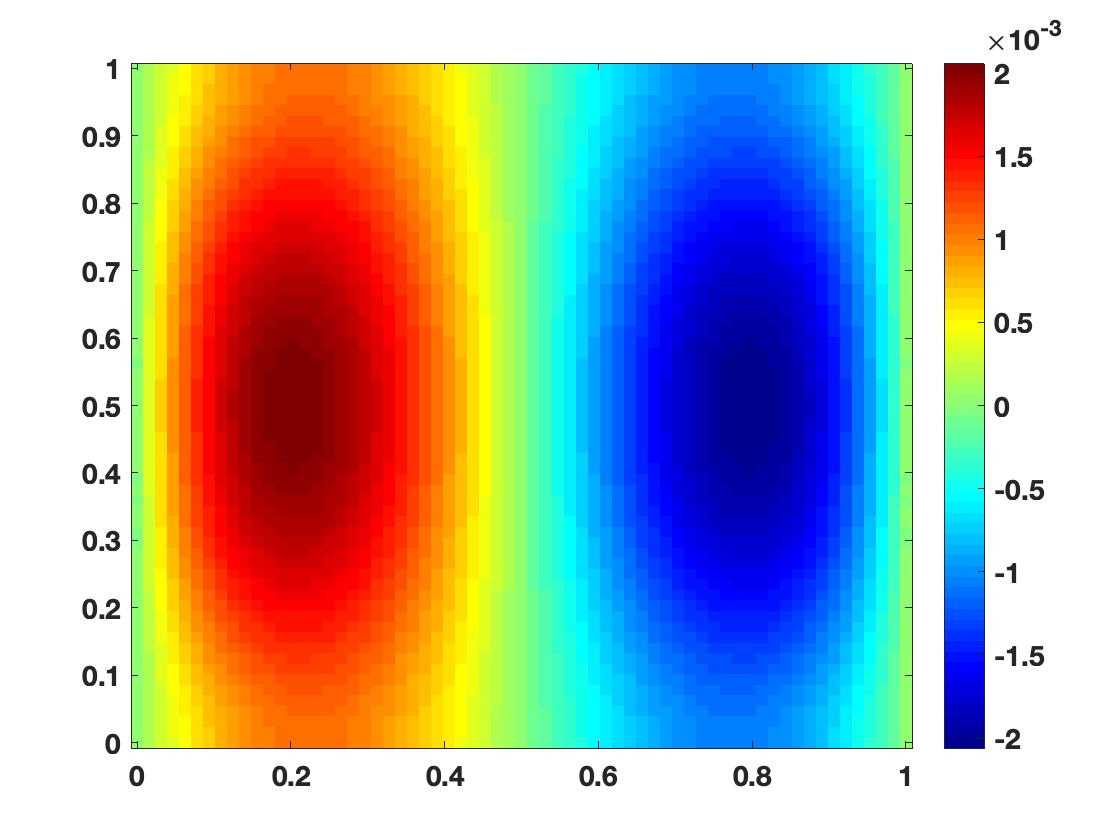}
  \includegraphics[width=0.45 \textwidth]{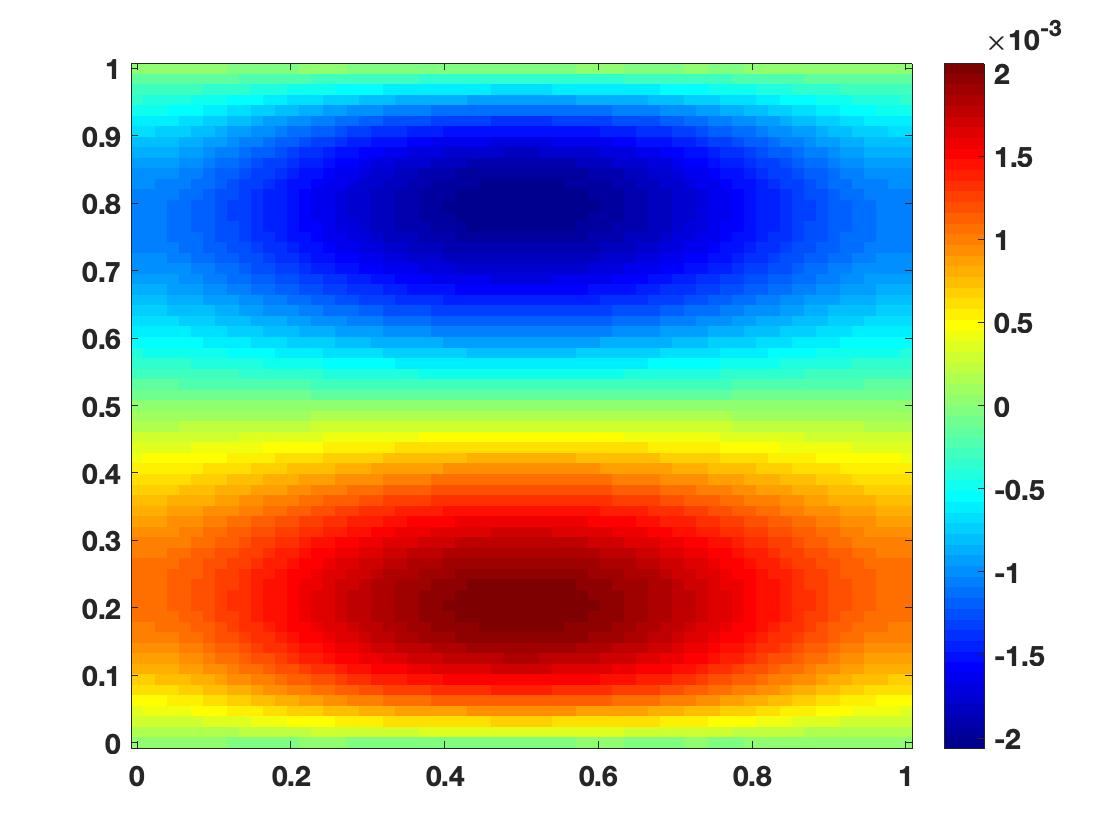}
  \caption{solutions to cell problem: $\chi_1(y_1,y_2)$  and $\chi_2(y_1,y_2)$ }
  \label{fig:chi_smooth}
\end{figure}

Utilizing the solutions to the cell problem \eqref{eq:cellPrb}, i.e., $\chi_1$ and $\chi_2$, we can obtain the effective coefficient $\kappa^*$ from \eqref{eqn:coefnext}, and then solve for $u_0$ by \eqref{eq:u0}.
Note that there is no microscale oscillation in the effective coefficient $\kappa^*$, thus the solution $u_0$ can be solved in a much coarser mesh compared to the mesh $\mathcal{T}_h$ associated to the original problem \eqref{eq:fine-model}. To simplify our notations, we adopt the same mesh $\mathcal{T}_h$ and finite element space $V_h^0$ as before, and utilize the conforming Galerkin formulation to solve for $u_{0,h}^{k+1}$ for $k=0,1,\cdots, N-1$.

We present the fine scale approximate solutions $u^{k}_{0,h}$ for $k=11,51 \text{ and } 101$ in Figure \ref{fig:u0}.
\begin{figure}[H]
  \centering
  \includegraphics[trim={2cm 0.5cm 2cm 0.5cm},clip,width=0.32 \textwidth]{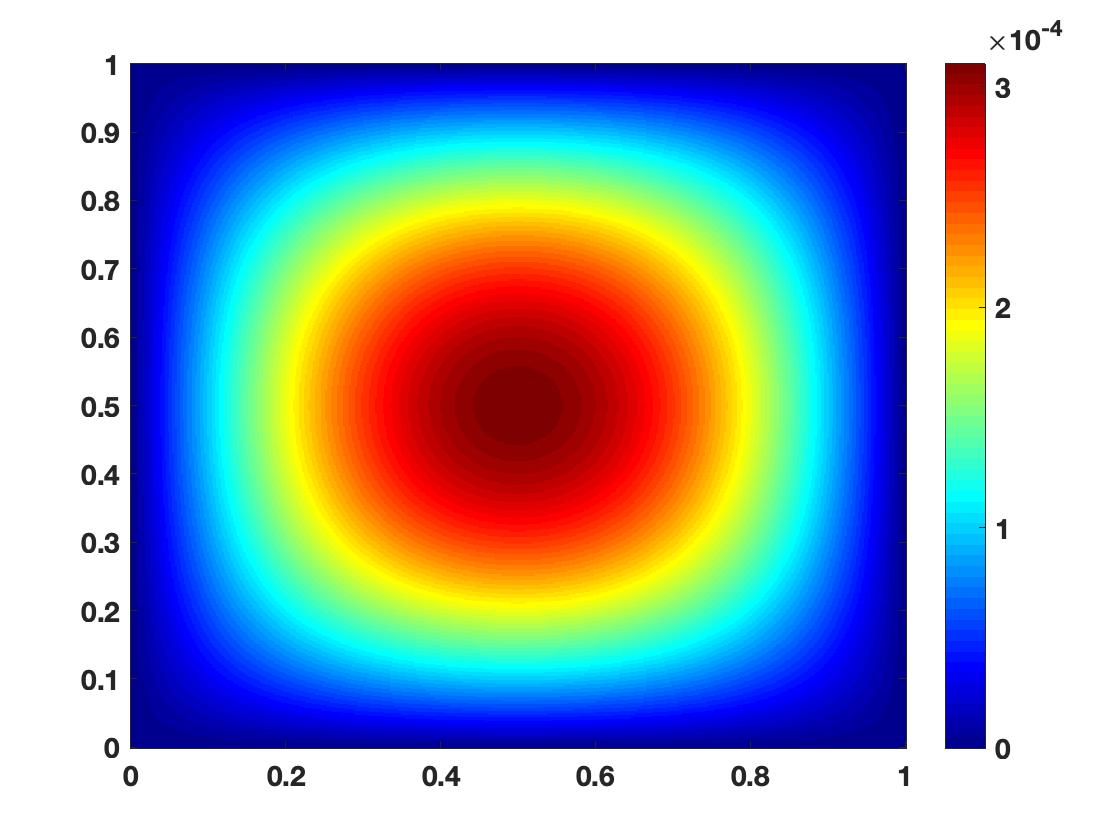}
  \includegraphics[trim={2cm 0.5cm 2cm 0.5cm},clip,width=0.32 \textwidth]{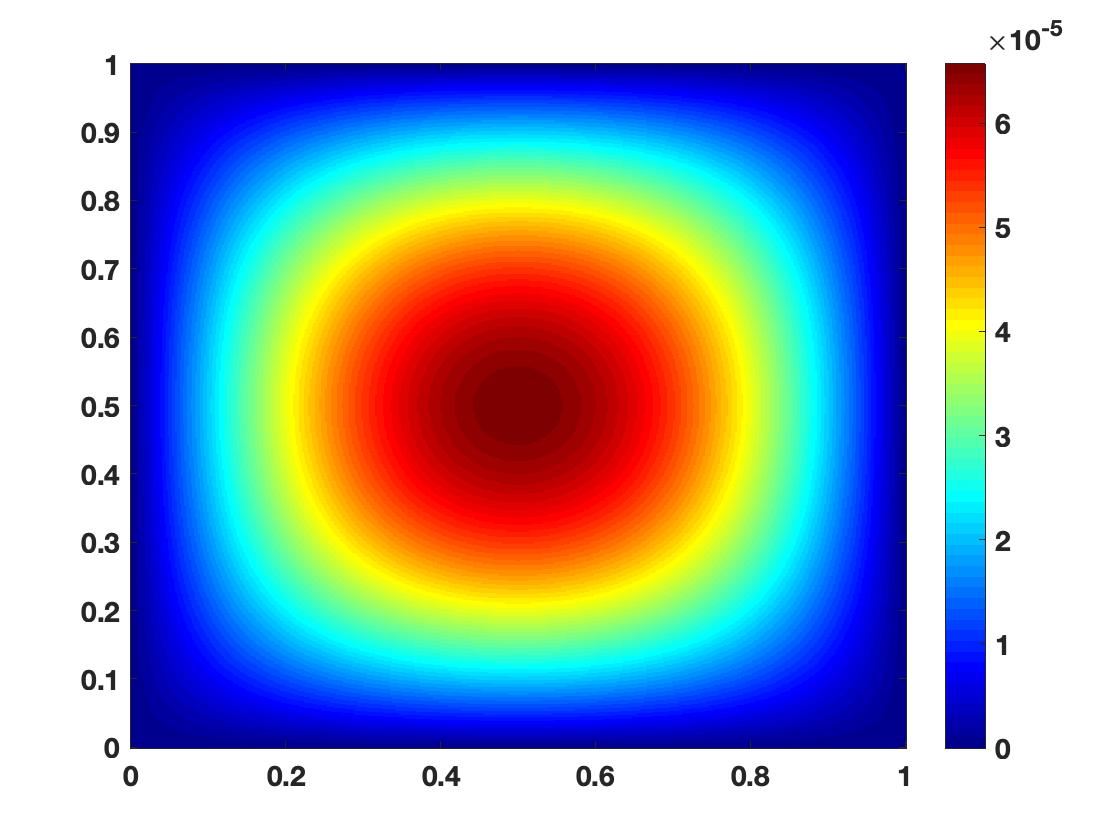}
  \includegraphics[trim={2cm 0.5cm 2cm 0.5cm},clip,width=0.32 \textwidth]{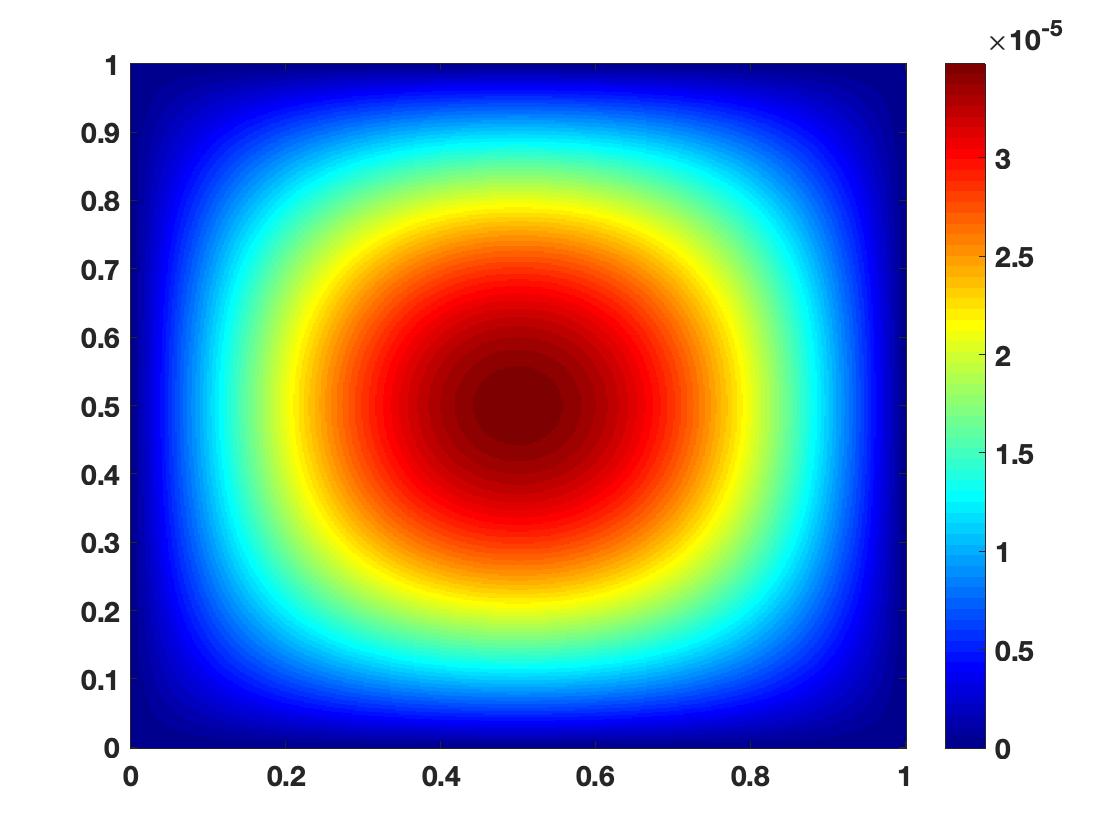}
  \caption{ The homogenized solution $u^{0,k}_h$ for $k=11,51 \text{ and } 101$ to Problem \eqref{eq:fine-model} with $\kappa$ in \eqref{eq:smoothPerm} and $\epsilon:=\frac{1}{8}$.}
  \label{fig:u0}
\end{figure}

Finally, the first order approximation $U^\epsilon_1$ can be estimated using formula (\ref{eq:firstapprox}). We present the graphs of the approximate solutions $U^{\epsilon,k}_{1,h}$ for $k=11,51 \text{ and } 101$ in Figure \ref{fig:u1}.

\begin{figure}[H]
  \centering
  \includegraphics[trim={2cm 0.5cm 2cm 0.5cm},clip,width=0.32 \textwidth]{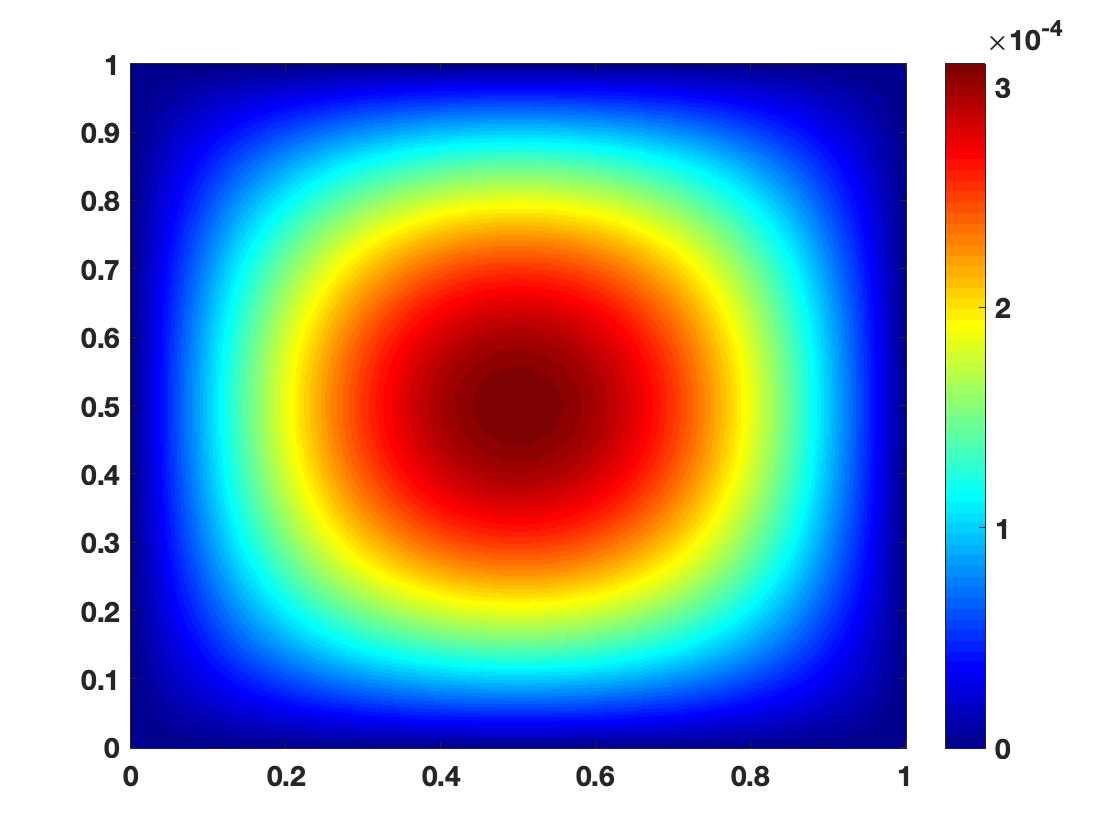}
  \includegraphics[trim={2cm 0.5cm 2cm 0.5cm},clip,width=0.32 \textwidth]{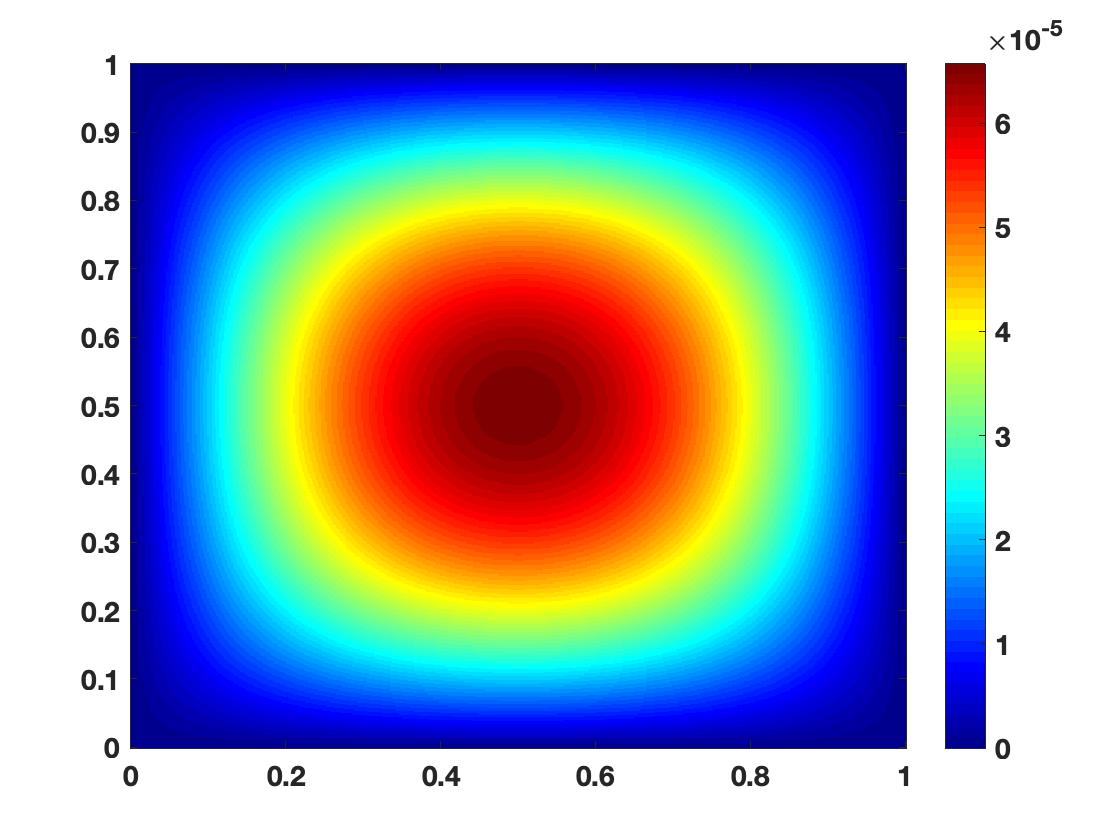}
  \includegraphics[trim={2cm 0.5cm 2cm 0.5cm},clip,width=0.32 \textwidth]{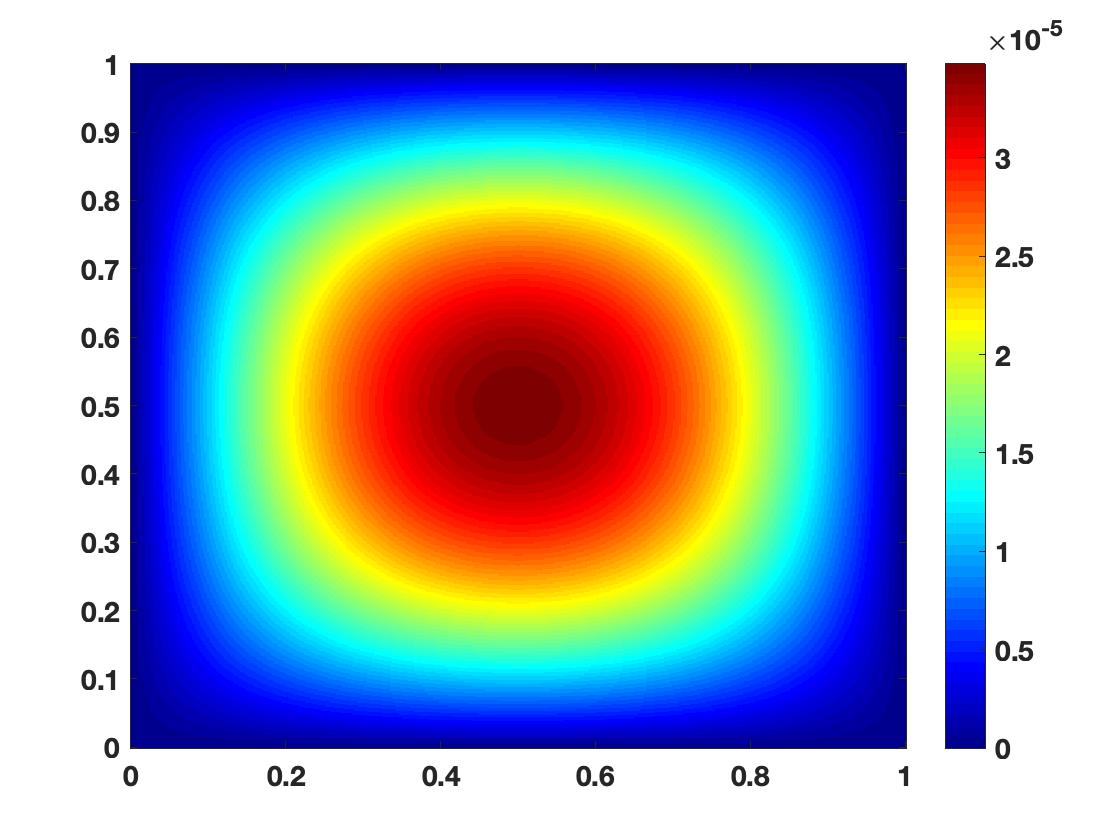}
  \caption{The first order approximation solution $U^{\epsilon,k}_{1,h}$ for $k=11,51 \text{ and } 101$ to Problem \eqref{eq:fine-model} with $\kappa$ in \eqref{eq:smoothPerm} and $\epsilon:=\frac{1}{8}$.}
  \label{fig:u1}
\end{figure}

We present the absolute error and relative error in $L^2$ norm and $H^1$ norm in Table \ref{error:SmoothPermeability}. The first column displays the discrete time steps at which we calculate the error.  The next two columns display the absolute error and relative error under the $L^2$ norm between the fine-scale solution $u_h^\epsilon$ and the first order approximation $U^\epsilon_{1,h}$. The last two columns display the absolute error and relative error under the $H^1$ norm between the fine-scale solution $u_h^\epsilon$ and the first order approximation $U^\epsilon_{1,h}$.

\begin{table}[H]
\begin{center}
\begin{tabular}{|c|c|c|c|c|}
        \hline
$t$ & $ \norm{u_{h}^\epsilon-U^\epsilon _{1,h}}_{L^2(D)}$& $\frac{ \norm{u_{h}^\epsilon-U^\epsilon _{1,h}}_{L^2(D)}}{ \norm{u_{h}^\epsilon} _{L^2(D) }}$& $\norm{u_{h}^\epsilon-U^\epsilon _{1,h}}_{H^1(D)}$& $\frac{ \norm{u_{h}^\epsilon-U^\epsilon _{1,h}}_{H^1(D)}}{ \norm{u^\epsilon_{h}} _{H^1(D) }}$
 \\ \hline
      0.1 & 2.1235e-9  &   1.3488e-5 &  1.4463e-7  &  2.0673e-4   \\
      0.5 & 4.5471e-10   & 1.3637e-5  &  3.0813e-8  &  2.0794e-4  \\
      1 & 2.411e-10   &  1.3656e-5  & 1.6325e-8 &  2.0809e-4    \\ \hline
\end{tabular}
\end{center}
\caption{The convergence history of the first order approximation to 
Problem \eqref{eq:fine-model} with $\kappa$ in \eqref{eq:smoothPerm} and $\epsilon:=\frac{1}{8}$.}
 \label{error:SmoothPermeability}
\end{table}

Furthermore, one numerical experiment is conducted with larger variation in the coefficient $\kappa$ compared to \eqref{eq:smoothPerm} to see the influence of the variation on the accuracy of homogenization. In this experiment, we take the same initial data $a(x_1,x_2)$ and the period $\epsilon:=\frac{1}{8}$. We set
\begin{align}\label{eq:smoothPermh}
%\kappa^\epsilon(x_1,x_2):=10+9\sin\Bigg(2\pi\lfloor\frac{ x_1} {\eps}\rfloor\lfloor\frac{ x_2} {\eps}\rfloor\Big (1-\lfloor\frac{ x_1} {\eps}\rfloor\Big)\Big(1-\lfloor\frac{ x_2} {\eps}\rfloor\Big)\Bigg).
\kappa(y_1,y_2):=10+9\sin\Big(2\pi\{{ y_1} \}\{{ y_2}\}\big (1-\{ y_1\}\big)\big(1-\{ y_2\}\big)\Big).
\end{align}
Note that the variation in this coefficient $\kappa^{\eps}$ is much larger than that defined in \eqref{eq:smoothPerm}. The convergence history of the first order approximation with $\kappa(y_1,y_2)$ in \eqref{eq:smoothPermh} is presented in Table \ref{table:error:SmoothPermeability_BiggerContrast}. Compared with the results in Table \ref{error:SmoothPermeability}, a larger variation in the diffusion coefficient $\kappa(y_1,y_2)$ results in a larger error in the first order approximation. 
\begin{table}[H]
\begin{center}
\begin{tabular}{|c|c|c|c|c|}
        \hline
$t$ & $ \norm{u_{h}^\epsilon-U^\epsilon _{1,h}}_{L^2(D)}$& $\frac{ \norm{u_{h}^\epsilon-U^\epsilon _{1,h}}_{L^2(D)}}{ \norm{u_{h}^\epsilon} _{L^2(D) }}$& $\norm{u_{h}^\epsilon-U^\epsilon _{1,h}}_{H^1(D)}$& $\frac{ \norm{u_{h}^\epsilon-U^\epsilon _{1,h}}_{H^1(D)}}{ \norm{u^\epsilon_{h}} _{H^1(D) }}$
 \\ \hline
      0.1 & 1.6061e-8  &   1.16596e-4 &  1.0832e-6 &  1.7659e-3   \\
      0.5 & 3.4581e-9  &  1.1753e-4  &  2.3239e-7 &  1.7737e-3   \\
      1 & 1.8342e-9   &  1.1765e-4  & 1.2321e-7 &  1.7747 e-3   \\ \hline
\end{tabular}
\end{center}
\caption{The convergence history of the first order approximation to Problem \eqref{eq:fine-model} with $\kappa$ in \eqref{eq:smoothPermh} and $\epsilon:=\frac{1}{8}$.}
        \label{table:error:SmoothPermeability_BiggerContrast}
\end{table}
In order to verify the convergence rate of the first approximation $U_1^{\epsilon}$, we test two different values of the parameter $\epsilon$ being $\frac{1}{16}$ and $\frac{1}{32}$, respectively. Their corresponding convergence histories are shown in Tables \ref{table:error:SmoothPermeability_SmallerEps} and \ref{table:error:SmoothPermeability_BiggerEps}. %The error results show that as $\epsilon$ decreases, the error tends to decrease, as expected from Corollary \ref{coro:1order}.

\begin{table}[H]
\begin{center}
\begin{tabular}{|c|c|c|c|c|}
        \hline
$t$ & $ \norm{u_{h}^\epsilon-U^\epsilon _{1,h}}_{L^2(D)}$& $\frac{ \norm{u_{h}^\epsilon-U^\epsilon _{1,h}}_{L^2(D)}}{ \norm{u_{h}^\epsilon} _{L^2(D) }}$& $\norm{u_{h}^\epsilon-U^\epsilon _{1,h}}_{H^1(D)}$& $\frac{ \norm{u_{h}^\epsilon-U^\epsilon _{1,h}}_{H^1(D)}}{ \norm{u^\epsilon_{h}} _{H^1(D) }}$
 \\ \hline
      0.1 & 5.2208e-10  &   3.3162e-6 &  7.1850e-8  &  1.0270e-4   \\
      0.5 & 1.1117e-10   &  3.3341e-6  &  1.5262e-8  &  1.0300e-4   \\
      1 & 5.8895e-11   &  3.3363e-6  & 8.0828e-9 &  1.0303e-4    \\ \hline
\end{tabular}
\end{center}
\caption{The convergence history of the first order approximation to 
Problem \eqref{eq:fine-model} with $\kappa$ in \eqref{eq:smoothPerm} and $\epsilon:=\frac{1}{16}$.}
        \label{table:error:SmoothPermeability_SmallerEps}
\end{table}

\begin{table}[H]
\begin{center}
\begin{tabular}{|c|c|c|c|c|}
        \hline
$t$ & $ \norm{u_{h}^\epsilon-U^\epsilon _{1,h}}_{L^2(D)}$& $\frac{ \norm{u_{h}^\epsilon-U^\epsilon _{1,h}}_{L^2(D)}}{ \norm{u_{h}^\epsilon} _{L^2(D) }}$& $\norm{u_{h}^\epsilon-U^\epsilon _{1,h}}_{H^1(D)}$& $\frac{ \norm{u_{h}^\epsilon-U^\epsilon _{1,h}}_{H^1(D)}}{ \norm{u^\epsilon_{h}} _{H^1(D) }}$
 \\ \hline
      0.1 & 1.4377e-10  &   9.1322e-7 &  3.8067e-8  &  5.4412e-5   \\
      0.5 & 3.0547e-11   &  9.1614e-7  &  8.0778e-9  &  5.4515e-5  \\
      1 & 1.6179e-11   &  9.1649e-7& 4.2776e-9 &  5.4527e-5    \\ \hline
\end{tabular}
\end{center}
\caption{The convergence history of the first order approximation to 
Problem \eqref{eq:fine-model} with $\kappa$ in \eqref{eq:smoothPerm} and $\epsilon:=\frac{1}{32}$.}
        \label{table:error:SmoothPermeability_BiggerEps}
\end{table}

One can calculate directly from Tables \ref{error:SmoothPermeability}, \ref{table:error:SmoothPermeability_SmallerEps} and \ref{table:error:SmoothPermeability_BiggerEps} that the first order approximation maintains a convergence rate of $\mathcal{O}(\epsilon^{0.9623})$, $\mathcal{O}(\epsilon^{0.9657})$ and $\mathcal{O}(\epsilon^{0.9661})$ for $t=0.1$, $t=0.5$ and $t=1$, respectively. 

We also test different values of the parameter $\alpha\in (0,1)$ and they all exhibit a similar convergence rate, as proved in Corollary \ref{coro:1order}. For the brevity of presentation, we will not present these results. 

\subsection{Numerical tests with non-smooth permeability fields}\label{subsec:nonsmooth}
Even though the theoretical result presented in Corollary \ref{coro:1order} is proved under the assumption that the diffusion coefficient $\kappa(y_1,y_2)$ is sufficiently smooth, we investigate in this section how well the first order approximation $U_1^{\epsilon}(x,t)$ performs when the coefficient $\kappa(y_1,y_2)$ is rough and can admit large variation. 

Firstly, we define two rough permeability fields $\kappa(x_1,x_2)$ with different variations. 
Let $U:=[\frac{1}{5},\frac{4}{5}]^2$ be a rectangle. Recall that $Y=[0,1]^2$ is a unit square. We set the variation in the first permeability field to be 1.1, which is defined by  
\begin{equation}\label{eq:kappa3}
\kappa_1(x_1,x_2):=
\left\{
\begin{aligned}
&11, &&\text{ if } (x_1,x_2)\in U\\
&10, &&\text{ if } (x_1,x_2)\in Y\setminus U. 
\end{aligned}
\right. 
\end{equation}
See Figure \ref{fig:NonSmoothPermeability_LessContrast} for an illustration.
\begin{figure}[H]
  \centering
  \includegraphics[trim={2cm 0.5cm 2cm 0.5cm},clip,width=0.4 \textwidth]{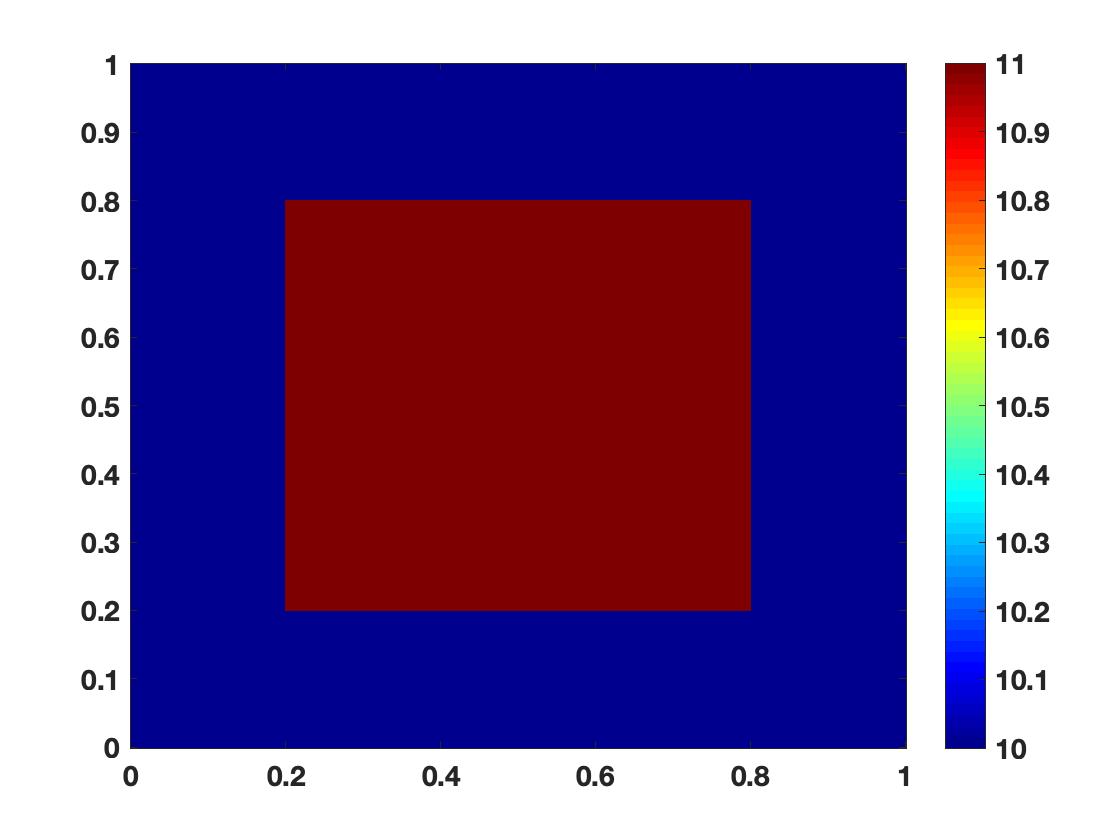}
  \includegraphics[trim={2cm 0.5cm 2cm 0.5cm},clip,width=0.4 \textwidth]{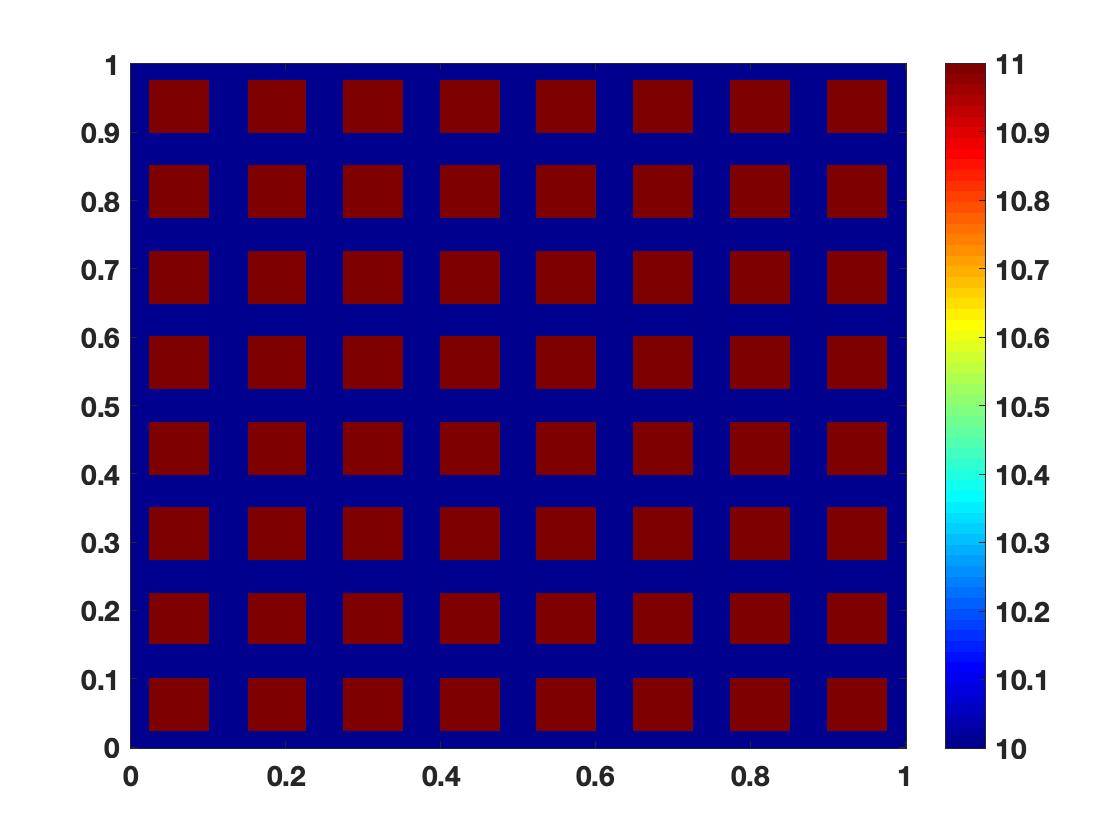}
  \caption{A nonsmooth permeability field of smaller variation: $\kappa_1(x_1,x_2)$ and $\kappa_1^\epsilon(x_1,x_2)$ with $\epsilon:=\frac{1}{8}$.}
  \label{fig:NonSmoothPermeability_LessContrast}
\end{figure}
We set the variation of the second rough permeability field to 2 and define
\begin{equation}\label{eq:kappa4}
\kappa_2(x_1,x_2):=
\left\{
\begin{aligned}
&20, &&\text{ if } (x_1,x_2)\in U\\
&10, &&\text{ if } (x_1,x_2)\in Y\setminus U. 
\end{aligned}
\right. 
\end{equation}
We present the graphs of such $\kappa(y_1,y_2)$ and $\kappa^\epsilon(x_1,x_2)$ in Figure \ref{fig:NonSmoothPermeability_BiggerContrast}.
\begin{figure}[H]
  \centering
  \includegraphics[trim={2cm 0.5cm 2cm 0.5cm},clip,width=0.4 \textwidth]{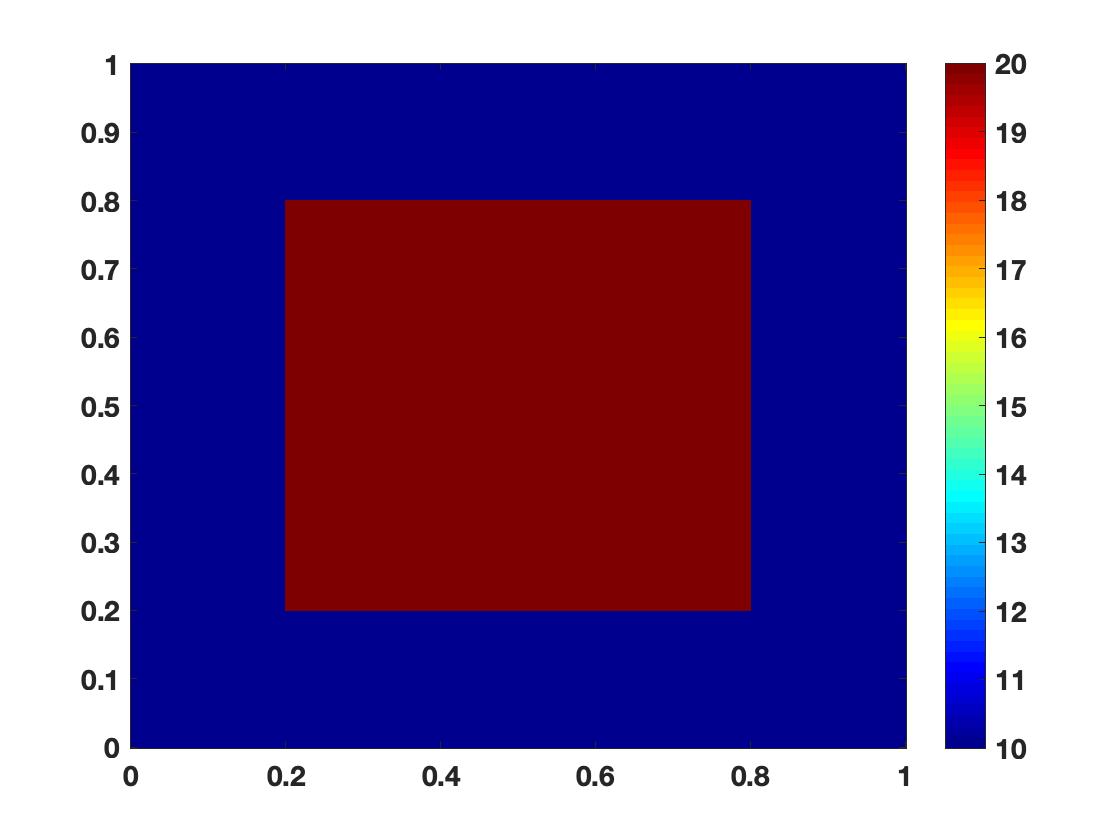}
  \includegraphics[trim={2cm 0.5cm 2cm 0.5cm},clip,width=0.4 \textwidth]{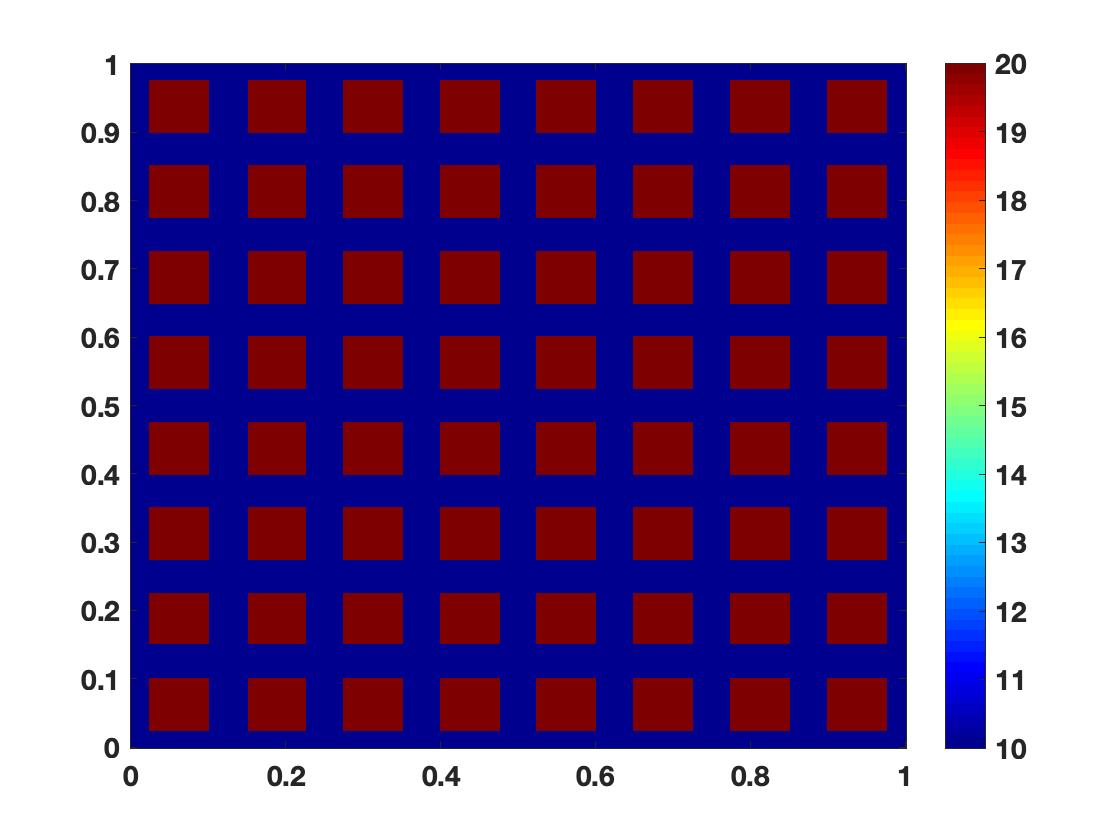}
  \caption{A nonsmooth permeability field of larger variation: $\kappa_2(x_1,x_2)$  and $\kappa_2^\epsilon(x_1,x_2)$ with $\epsilon=\frac{1}{8}$. }
  \label{fig:NonSmoothPermeability_BiggerContrast}
\end{figure}

Let the parameter $\epsilon:=1/8$. We present the convergence histories of the first order approximation $U_1^{\epsilon}(x,t)$ with those two nonsmooth permeability fields in \eqref{eq:kappa3} and \eqref{eq:kappa4} in Tables \ref{table:error:NonsmoothPermeability_LessContrast} and \ref{table:error:NonsmoothPermeability_BiggerContrast}. One can observe that the former outperforms the latter. 

\begin{table}[H]
\begin{center}
\begin{tabular}{|c|c|c|c|c|}
        \hline
$t$ & $ \norm{u_{h}^\epsilon-U^\epsilon _{1,h}}_{L^2(D)}$& $\frac{ \norm{u_{h}^\epsilon-U^\epsilon _{1,h}}_{L^2(D)}}{ \norm{u_{h}^\epsilon} _{L^2(D) }}$& $\norm{u_{h}^\epsilon-U^\epsilon _{1,h}}_{H^1(D)}$& $\frac{ \norm{u_{h}^\epsilon-U^\epsilon _{1,h}}_{H^1(D)}}{ \norm{u^\epsilon_{h}} _{H^1(D) }}$
 \\ \hline
      0.1 & 8.3877e-10  &   2.5129e-2 &  5.6728e-9  &  3.8002e-2   \\
      0.5 & 1.8973e-10   &  2.5130e-2 &  1.2832e-9  &  3.8003e-2   \\
      1 & 1.0121e-10   &  2.5130e-2 & 6.8450e-10  &  3.8003e-2   \\ \hline
\end{tabular}
\end{center}
\caption{The convergence history of the first order approximation to 
Problem \eqref{eq:fine-model} with $\kappa:=\kappa_1$ in \eqref{eq:kappa3} and $\epsilon:=\frac{1}{8}$.}
        \label{table:error:NonsmoothPermeability_LessContrast}
\end{table}

\begin{table}[H]
\begin{center}
\begin{tabular}{|c|c|c|c|c|}
        \hline
$t$ & $ \norm{u_{h}^\epsilon-U^\epsilon _{1,h}}_{L^2(D)}$& $\frac{ \norm{u_{h}^\epsilon-U^\epsilon _{1,h}}_{L^2(D)}}{ \norm{u_{h}^\epsilon} _{L^2(D) }}$& $\norm{u_{h}^\epsilon-U^\epsilon _{1,h}}_{H^1(D)}$& $\frac{ \norm{u_{h}^\epsilon-U^\epsilon _{1,h}}_{H^1(D)}}{ \norm{u^\epsilon_{h}} _{H^1(D) }}$
 \\ \hline
      0.1 & 5.5466e-9  &   1.6618e-1&  3.6977e-8  &  2.4771e-1 \\
      0.5 & 1.2546e-9   &  1.6618e-1  &  8.3640e-9  & 2.4772e-1   \\
      1 & 6.6929e-10 &  1.6618e-1 & 4.4618e-9  &  2.4772e-1   \\ \hline
\end{tabular}
\end{center}
\caption{The convergence history of the first order approximation to 
Problem \eqref{eq:fine-model} with $\kappa:=\kappa_2$ in \eqref{eq:kappa4} and $\epsilon:=\frac{1}{8}$.}
        \label{table:error:NonsmoothPermeability_BiggerContrast}
\end{table}

\subsection{Numerical tests with rough initial data}\label{subsec:initial_L2}
Now we study the convergence rate of the first order approximation $U^{\epsilon}_1(x,t)$ when the initial data $a(x)$ is nonsmooth. To this aim, we take a rough initial data $a(x_1,x_2)$ defined by  
\begin{equation}\label{eq:rough-a}
a(x_1,x_2):=
\left\{
\begin{aligned}
&1, &&\text{ if } (x_1,x_2)\in (0.5,1)^2\\
&0, &&\text{ if } (x_1,x_2)\in (0,1)^2\setminus (0.5,1)^2, 
\end{aligned}
\right. 
\end{equation}
which is depicted in Figure \ref{fig:initial_data_L2}. 
\begin{figure}[H]
  \centering
  \includegraphics[trim={2cm 0.5cm 2cm 0.5cm},clip,width=0.4 \textwidth]{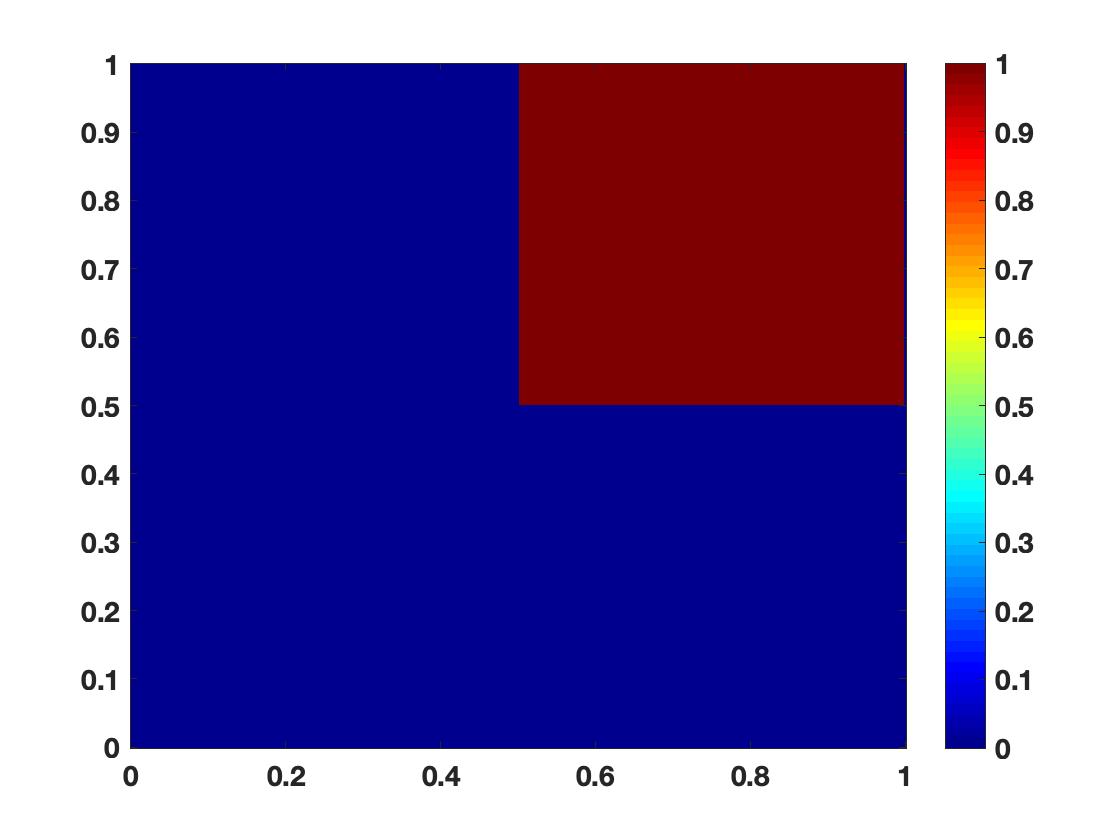}
  \caption{The initial data $a(x_1,x_2)$ as defined in \eqref{eq:rough-a}.}
  \label{fig:initial_data_L2}
\end{figure}
To emphasize the effect of a rough initial data on the convergence rate of the first order approximation, we take the smooth permeability field $\kappa(x_1,x_2)$ as defined in \eqref{eq:smoothPerm}. We adopt the same numerical scheme to calculate the numerical solutions as in Section \ref{subsec:smooth}. 
Like before, we test $\epsilon=\frac{1}{8}, \frac{1}{16}$ and $\frac{1}{32}$ to validate the convergence rate proved in Corollary \ref{coro:1order}. Their corresponding convergence histories of the first order approximation to Problem \eqref{eq:fine-model} are presented in Tables
\ref{table:error:a2InL2(D) smoothPermeability_BiggerContrast_8}, \ref{table:error:a2InL2(D) smoothPermeability_BiggerContrast_16} and \ref{table:error:a2InL2(D) smoothPermeability_BiggerContrast_32}.

\begin{table}[H]
\begin{center}
\begin{tabular}{|c|c|c|c|c|}
        \hline
$t$ & $ \norm{u_{h}^\epsilon-U^\epsilon _{1,h}}_{L^2(D)}$& $\frac{ \norm{u_{h}^\epsilon-U^\epsilon _{1,h}}_{L^2(D)}}{ \norm{u_{h}^\epsilon} _{L^2(D) }}$& $\norm{u_{h}^\epsilon-U^\epsilon _{1,h}}_{H^1(D)}$& $\frac{ \norm{u_{h}^\epsilon-U^\epsilon _{1,h}}_{H^1(D)}}{ \norm{u^\epsilon_{h}} _{H^1(D) }}$
 \\ \hline
      0.1 & 2.0111e-7  &   1.8092e-4 &  5.1832e-6  &  8.3829e-4  \\
      0.5 & 4.4356e-8  &  1.8618e-4  &  1.1545e-6  & 8.6128e-4  \\
      1 & 2.3592e-8 &  1.8678e-4 &6.1475e-7  &  8.6391e-4   \\ \hline
\end{tabular}
\end{center}
\caption{The convergence history of the first order approximation to 
Problem \eqref{eq:fine-model} with $a(x)$ defined in \eqref{eq:rough-a}, $\kappa$ in \eqref{eq:smoothPerm} and $\epsilon:=\frac{1}{8}$.}
        \label{table:error:a2InL2(D) smoothPermeability_BiggerContrast_8}
\end{table}
Comparing the results presented in Table \ref{table:error:a2InL2(D) smoothPermeability_BiggerContrast_8} with that in Table \ref{error:SmoothPermeability}, one can observe that the latter admits a first order approximation with a much smaller error as expected, since a rough initial data produces singularity in the solution when the time $t$ is small. 
\begin{table}[H]
\begin{center}
\begin{tabular}{|c|c|c|c|c|}
        \hline
$t$ & $ \norm{u_{h}^\epsilon-U^\epsilon _{1,h}}_{L^2(D)}$& $\frac{ \norm{u_{h}^\epsilon-U^\epsilon _{1,h}}_{L^2(D)}}{ \norm{u_{h}^\epsilon} _{L^2(D) }}$& $\norm{u_{h}^\epsilon-U^\epsilon _{1,h}}_{H^1(D)}$& $\frac{ \norm{u_{h}^\epsilon-U^\epsilon _{1,h}}_{H^1(D)}}{ \norm{u^\epsilon_{h}} _{H^1(D) }}$
 \\ \hline
      0.1 & 5.3061e-8  &   4.7739e-5 &  2.7039e-6  &  4.3744e-4  \\
      0.5 & 1.1742e-8  &  4.9292e-5  &  6.0373e-7  & 4.5053e-4   \\
      1 & 6.24783e-9 &  4.9471e-5 &3.2156e-7  &  4.5204e-4   \\ \hline
\end{tabular}
\end{center}
\caption{The convergence history of the first order approximation to 
Problem \eqref{eq:fine-model} with $a(x)$ defined in \eqref{eq:rough-a}, $\kappa$ in \eqref{eq:smoothPerm} and $\epsilon:=\frac{1}{16}$.}
        \label{table:error:a2InL2(D) smoothPermeability_BiggerContrast_16}
\end{table}

\begin{table}[H]
\begin{center}
\begin{tabular}{|c|c|c|c|c|}
        \hline
$t$ & $ \norm{u_{h}^\epsilon-U^\epsilon _{1,h}}_{L^2(D)}$& $\frac{ \norm{u_{h}^\epsilon-U^\epsilon _{1,h}}_{L^2(D)}}{ \norm{u_{h}^\epsilon} _{L^2(D) }}$& $\norm{u_{h}^\epsilon-U^\epsilon _{1,h}}_{H^1(D)}$& $\frac{ \norm{u_{h}^\epsilon-U^\epsilon _{1,h}}_{H^1(D)}}{ \norm{u^\epsilon_{h}} _{H^1(D) }}$
 \\ \hline
      0.1 & 1.3154e-8  &   1.1835e-5 &  1.3839e-6  &  2.2390e-4  \\
      0.5 & 2.9144e-9   &  1.2235e-5  &  3.0918e-7  & 2.3074e-4   \\
      1 & 1.5510e-9 &  1.2281e-5 & 1.6469e-7  &  2.3153e-4   \\ \hline
\end{tabular}
\end{center}
\caption{The convergence history of the first order approximation to 
Problem \eqref{eq:fine-model} with $a(x)$ defined in \eqref{eq:rough-a}, $\kappa$ in \eqref{eq:smoothPerm} and $\epsilon:=\frac{1}{32}$.}
        \label{table:error:a2InL2(D) smoothPermeability_BiggerContrast_32}
\end{table}
One can calculate from Tables \ref{table:error:a2InL2(D) smoothPermeability_BiggerContrast_8}, \ref{table:error:a2InL2(D) smoothPermeability_BiggerContrast_16} and \ref{table:error:a2InL2(D) smoothPermeability_BiggerContrast_32} that the first order approximation has a convergence rate of $\mathcal{O}(\epsilon^{0.9523})$, $\mathcal{O}(\epsilon^{0.9502})$ and $\mathcal{O}(\epsilon^{0.9499})$ for $t=0.1$, $t=0.5$ and $t=1$, respectively. 

\section{Conclusion}\label{sec:conclusion}
This paper is concerned with constructing a homogenization theorem for the time-fractional diffusion equations with a homogeneous Dirichlet boundary condition and an inhomogeneous initial data $a(x)\in L^{2}(D)$ in a bounded convex polyhedral domain $D$. Under the assumption that the diffusion coefficient $\kappa^{\epsilon}(x)$ is smooth and periodic with a period $\epsilon>0$ being sufficiently small, we proved that its first order approximation has a convergence rate of $\mathcal{O}(\epsilon^{1/2})$ when the dimension $d\leq 2$ and $\mathcal{O}(\epsilon^{1/6})$ when $d=3$. Several numerical tests are presented to show the performance of the first order approximation. Future studies include deriving the homogenization results for the case of existence of extra time scale in the diffusion coefficient $\kappa^{\epsilon,\eta}(x,t)$ with $\epsilon$ and $\eta$ being small parameters. 
\section*{Acknowledgement}
GL acknowledges the support from the Royal Society through a Newton International Fellowship.
\bibliographystyle{siam}

\bibliography{ref}

\end{document}